\documentclass[12pt]{amsart}
\usepackage{amsmath,amssymb,times,color}
\usepackage{graphicx}
\usepackage{fig4tex}
\usepackage{dsfont}
\definecolor{webred}{rgb}{0.75,0,0}
\definecolor{webgreen}{rgb}{0,0.75,0}
\usepackage[citecolor=webgreen,colorlinks=true,linkcolor=webred]{hyperref}
\usepackage{stmaryrd} 
\usepackage{mathrsfs}
\usepackage{fig4tex}

\newtheorem{thm}{Theorem}[section]

\newtheorem{hyp}[thm]{Assumption}

\theoremstyle{definition}
\newtheorem{defn}[thm]{Definition}

\theoremstyle{remark}
\newtheorem{rem}[thm]{Remark}

\oddsidemargin=\evensidemargin

\numberwithin{equation}{section}

\renewcommand{\Re}{\operatorname{Re}}

\newcommand{\Div}{\operatorname{\mathrm{div}}}

\newcommand{\rot}{\operatorname{\mathrm{curl}}}

\newcommand{\db}[1]{_{\raise-0.3ex\hbox{$\scriptstyle #1$}}}
\newcommand{\dd}[1]{_{\raise-1.5pt\hbox{$\scriptstyle #1$}}}

\newcommand{\di}{\displaystyle}

\newcommand{\dr}{{\rm d}}

\newcommand{\iso}{_{\sf +}}
\newcommand{\con}{_{\sf -}}

\newcommand  {\N}{{\mathbb N}}

\newcommand  {\R}{{\mathbb R}}

\newcommand {\Id}{\mathbb {I}}

\newcommand{\BB}{\boldsymbol{\mathsf B}}
\newcommand  {\EE}{\boldsymbol{\mathsf E}}

\newcommand  {\HH}{\boldsymbol{\mathsf H}}

\newcommand  {\LL}{\boldsymbol{\mathsf L}}
\renewcommand  {\L}{{\mathrm L}}

\newcommand  {\QQ}{\boldsymbol{\mathsf Q}}

\newcommand  {\RR}{\boldsymbol{\mathsf R}}

\newcommand  {\VV}{\,\underline{\!{\boldsymbol{\mathfrak H}}\!}\,}
\newcommand  {\WW}{\,\underline{\!{\boldsymbol{\mathfrak E}}\!}\,}
\newcommand  {\XX}{\boldsymbol{\mathsf X}}

\newcommand  {\mur}{{\boldsymbol{\mu_{r}}}}

\newcommand  {\g}{{\boldsymbol{\mathsf g}}}

\newcommand  {\jj}{{\boldsymbol{\mathsf j}}}
\newcommand  {\nn}{\boldsymbol{\mathsf n}}

\newcommand  {\uu}{\boldsymbol{\mathsf u}}

\newcommand  {\xx}{\boldsymbol{\mathsf x}}
\newcommand  {\yy}{\boldsymbol{\mathsf y}}

\newcommand  {\cH}{\mathcal{H}}
\newcommand  {\cC}{\mathcal{C}}

\newcommand  {\cL}{\mathcal{L}}
\newcommand  {\cO}{\mathcal{O}}
\newcommand  {\cU}{{\mathcal U}}

\newcommand  {\cV}{{\mathfrak H}}
\newcommand  {\cW}{{\mathfrak E}}

\newcommand  {\bH}{\mathbf{H}}
\newcommand  {\bL}{\mathbf{L}}

\newcommand  {\sE}{\mathsf{e}}

\newcommand  {\eps}{\mathsf{\varepsilon}}

\newcommand  {\sv}{\mathfrak{h}}
\newcommand  {\fke}{\mathfrak{e}}

\newcommand  {\gm}{\mathfrak{m}}

\newcommand  {\bsE}{\boldsymbol{\mathsf{e}}}

\newcommand  {\bs}{\underline \sigma}
\newcommand  {\bm}{\underline \mu}

\newcommand{\D}{\mathsf D}

\definecolor{mpurple}{rgb}{0.6,0,0.8}
\definecolor{myblue}{rgb}{0.,0.2,0.8}
\definecolor{mygreen}{rgb}{0,0.75,0.0}
\definecolor{mred}{rgb}{0.9,0,0}
\definecolor{mbrun}{rgb}{0.8,0.5,0}

\newcommand{\Bk}{\color{black}}

\begin{document}

\title[Asymptotic modelling of a skin effect in magnetic conductors]
{Asymptotic modelling of a skin effect in magnetic conductors}

\author{Victor P\'eron}
\address{Laboratoire LMAP, UMR CNRS 5142,
Universit\'e de Pau et des Pays de l'Adour, 
64013 PAU cedex, France}

\begin{abstract}
We consider the time-harmonic Maxwell equations set on a domain made of two subdomains $\Omega_{-}$ and $\Omega_{+}$, such that $\Omega_{-}$ represents a magnetic conductor and  $\Omega_{+}$ represents a non-magnetic material, and the relative magnetic permeability  $\mu_{r}$  between the two materials is very high. Assuming smoothness for the interface between the subdomains and regularity of the data, the electric field solution of the Maxwell equations possesses an asymptotic expansion in powers of the parameter $\eps=1/ \sqrt{\mu_{r}}$ with profile terms rapidly decaying inside the magnetic conductor. We make explicit the first terms of this expansion. As an application of the asymptotic expansion we obtain the asymptotic behavior of a skin depth function that allows to measure the boundary layer phenomenon at large relative permeability. As another application of this expansion  we give elements of proof for the derivation of impedance boundary conditions (IBCs) up to  the third order of approximation with respect to the parameter $\varepsilon$, and we prove error estimates for the IBCs. 
\end{abstract}

\date{\today, Version 1}

\maketitle

\vskip0.5cm

\tableofcontents

\section{Introduction}

This work is concerned with the time-harmonic Maxwell equations in linear materials presenting high contrast of magnetic permeabilities. The  domain of interest is made up of two subdomains that represent a magnetic conductor and a non-magnetic material, and we assume that the relative magnetic permeability $\mu_{r}$  between the materials is very high. In this paper we pursue mathematical efforts that we undertook  in a  previous work \cite{peron2024uniform} on asymptotic modelling of  a {\it skin effect} that occurs close to the surface of a magnetic conductor  in the context above. 

Assuming smoothness for the interface between the subdomains, we first prove that the electric field solution of the Maxwell equations possesses a multiscale expansion in powers of  $\eps=\dfrac1{\sqrt{\mu_{r}}}$ with profile terms rapidly decaying inside the magnetic conductor, cf. Sec. \ref{S3}, and also Sec. \ref{AppA}. This expansion allows to describe a boundary layer phenomenon which is different from the anomalous  skin effect that occurs at high conductivity in non-magnetic materials, see, {\it e.g.} \cite{HJN08,CDFP11}. We provide the first terms of this expansion in  Section \ref{S3.1}. 
We validate this expansion by proving error estimates (i.e. estimates for remainders) in Section \ref{Svalid-bis}. The proof is based on uniform a priori estimates for Maxwell transmission problem (cf. \cite[Th. 2.1]{peron2024uniform}) that we remind in Section \ref{S2}. We make also explicit the first terms of this expansion up to the third order of approximation with respect to $\eps$  in the non-magnetic region, and we derive the first profiles up to the second order of approximation  in the magnetic part. 

 The multiscale expansion for the electric field as well as error estimates for this expansion are new results. In \cite{peron2024uniform} we derived rigorously a multiscale expansion for the magnetic field  in powers of  $\eps$ for this problem, and we made explicit the first terms of this expansion. Several works are devoted also to asymptotic expansions of the electromagnetic field in a domain made of two subdomains {\em when the interface is smooth}: see {\em e.g.} \cite{S83,MS84,MS85,PePo21,APPK21,AP23} for plane interface and eddy current approximations   at high conductivity or at large relative magnetic permeability, and \cite{HJN08,CDP09,DFP10,CDFP11} for a three-dimensional model of skin effect in electromagnetism at high conductivity. We refer also the reader to the work in Ref. \cite{BIJ20} on asymptotic modelling of skin effects in coaxial cables.

Then we introduce a ``skin depth'' function to measure the boundary layer phenomenon. As an application of the multiscale expansion for the electric field we obtain the asymptotic behavior of this function at large relative permeability, cf. Theorem  \ref{T1} in Sec. \ref{S3.2}:
\begin{equation*}
   \cL(\mu_{-},\sigma_{-},y_{\alpha})=
   \ell(\mu_{-},\sigma_{-}) \phi(\sqrt{{\omega\varepsilon_0} /{\sigma_{-}}}) \Big(1+\cH(y_{\alpha})\, \ell(\mu_{-},\sigma_{-})  \phi(\sqrt{{\omega\varepsilon_0} /{\sigma_{-}}}) + \cO(\mu_{r}^{-1})\Big) \, .
 \end{equation*}
Here $\mu_{-}$, and $\sigma_{-}$ represent the magnetic permeability and the conductivity of the magnetic conductor respectively, $\ell(\mu_{-},\sigma_{-})$ is the classical skin depth, $\phi$ is a scalar function defined by Eq. \eqref{Ephi} (cf. Sec. \ref{S3.2}) such that  $\phi(\delta)\to 1$ as $\delta\to0$ (large conductivity ($\sigma_{-}$)/low frequency  
 limit) and  $\phi(\delta)\sim \sqrt2 \delta$ as $\delta\to\infty$ (high frequency/small conductivity limit), $y_\alpha$, $\alpha=1,2$, are tangential coordinates on the interface $\Sigma$,  and $\cH(y_{\alpha})$ is the mean curvature of  $\Sigma$.   
Then we compare this result with several results on this topic in the literature \cite{DFP10,CDFP11,PePo21}, cf. Sec. \ref{S4.2}.   We refer also to the work in Ref. \cite{CDP19} where the authors measure another boundary layer phenomenon in spherical domains by introducing similarly a suitable characteristic function. 

 As another application of the asymptotic expansion for the electric field it is possible to derive impedance boundary conditions (IBCs) on the interface up to the third order of approximation with respect to the parameter $\varepsilon$, cf. Sec. \ref{S5}.  IBCs are useful for numerical purposes since they allow to reduce the computational domain, see, {\it e.g.} \cite{artola1992diffraction,lafitte1993equations,senior1995approximate,bendali1996effect,ammari1998effective,nedelec2001,Mo03,HJN08,CoDaNi10,hoppe2018impedance,auvray2018asymptotic,stupfel2021well} where additional references may also be found.  The new IBCs allow to solve simpler problems set in the subdomain representing  the non-magnetic material.  The first order IBC is the perfectly insulating electric boundary conditions whereas the second and third order IBCs are  Leontovich type boundary conditions, cf. Eqs. \eqref{EIBC1}, \eqref{EIBC3}, and \eqref{EIBC4b} in Sec. \ref{IBC}, respectively. The new second order IBC coincides with the classical Leontovich condition for strongly absorbing material (large conductivity limit), cf. Remark \ref{Leon}. In \cite{ECCOMAS24} we developed briefly the asymptotic method for  deriving these IBCs. In this work we give elements of proof  for the derivation of IBCs, and we prove optimal error  estimates for the approximate models with IBCs (cf. Theorem \ref{TValid}).  
  
  We begin this paper by introducing the mathematical model, cf. Section \ref{S2}. We recall uniform a priori estimates for Maxwell solutions of the interface problem at high relative magnetic permeability (Sec. \ref{S2.1}), and we introduce an electric formulation for this problem (Sec. \ref{S2.2}).

\section{The mathematical model}
\label{S2}

Let $\Omega$ be a smooth bounded simply connected domain in ${\R}^3$, and $\Omega_{-}\subset\subset\Omega$ be a Lipschitz connected subdomain of $\Omega$. 
We denote by $\Gamma$ the boundary of $\Omega$, and by  $\Sigma$  the boundary of $\Omega_{-}$. 
Finally, we denote by $\Omega_{+}$ the complementary of $\overline{\Omega}_{-}$ in $\Omega$, cf. Figure~\ref{F1}. 

We consider the time-harmonic Maxwell equations given by Faraday's and Amp\`ere's laws in $\Omega$:
\begin{equation}
    \rot \ \EE - i \omega\bm \HH = 0 \quad \mbox{and}\quad
    \rot \ \HH + (i\omega\varepsilon_0 - \bs) \EE =   \jj
    \quad\mbox{in}\quad\Omega\, .
\label{MS}
\end{equation}
Here, $(\EE,\HH)$ represents the electromagnetic field, $\varepsilon_0$ is the electric permittivity, $\omega$ is the angular frequency ($\omega\neq0$ and the time convention is $\exp(-i\omega t)$), $\jj$ represents a current density  and is supposed to belong to  $\bL^2(\Omega)$, $\bs$ is the electric conductivity, and  $\bm$ is the magnetic permeability. We assume that the domain $\Omega$ is made of two connected subdomains $\Omega_+$ and $\Omega_-$ in which the coefficients $\bs$ and $\bm$ take different values $(\sigma_+>0,\sigma_->0)$ and $(\mu_{+}>0, \mu_{-}>0)$, respectively, and we denote by $\mu_{r}$ the relative magnetic permeability between the  subdomains  $\Omega_-$ and $\Omega_+$, and which is defined as: 
\begin{equation}
\label{E0}
 \mu_{r}=\mu_{-}/\mu_{+} \ .
\end{equation}

\begin{figure}[ht]
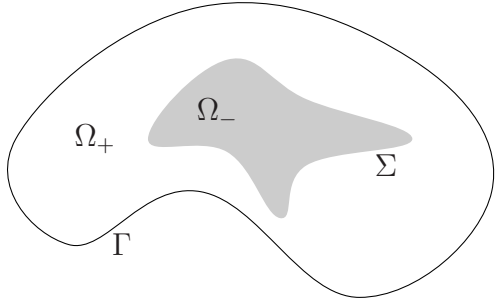

\begin{center}
\figinit{1.0pt}
\figpt 1:(-100,20)
\figpt 2:(-30,70)\figpt 3:(50,50)
\figpt 4:(80,0)\figpt 5:(30,-40)
\figpt 6:(-30,0)\figpt 7:(-80,-20)
\figpt 8:(-50,20)\figpt 9:(-20,50)\figpt 10:(50,20)\figpt11:(10,35)
\figpt 12:(10,10)\figpt 13:(0,-10)
\figpt 14:(-20,15)%
\figpt 16:(-70,20)\figpt 15:(-10,30)
\figpt 17:(40,10) \figpt 18:(-60,-20)
\psbeginfig{}
\pscurve[1,2,3,4,5,6,7,1,2,3]
\pssetfillmode{yes}\pssetgray{0.8}
\pscurve[8,9,11,10,12,13,14,8,9,11]
\pssetfillmode{no}\pssetgray{0}
\psendfig
\end{center}
\figvisu{\figBoxA}{
}{
\figwritew 15: $\Omega_{-}$(6pt)
\figwritec [16]{$\Omega_{+}$}
\figwritec [17]{$\Sigma$}
\figwritec [18]{$\Gamma$}
}
\centerline{\box\figBoxA}
\caption{The domain $\Omega$ and its subdomains $\Omega_-$ (magnetic conductor) and  $\Omega_+$  (non-magnetic material)}
\label{F1}
\end{figure}

To complement the Maxwell harmonic equations \eqref{MS}, we consider either the perfectly insulating electric boundary conditions 
\begin{equation}
\label{PIbc}
   \EE\cdot\nn=0 \quad\mbox{and}\quad \HH\times\nn=0\quad\mbox{on}\quad\Gamma\, ,
\end{equation}
or the perfectly conducting electric boundary conditions 
\begin{equation}
\label{Pcbc}
   \EE\times\nn=0 \quad\mbox{and}\quad \HH\cdot\nn=0\quad\mbox{on}\quad\Gamma\,.
\end{equation}

\subsection{Existence of solutions}
\label{S2.1}
Hereafter, we denote by  $\|\cdot \|_{0,\mathcal{O}}$ the norm in $\bL^2(\mathcal{O})$. We quote from \cite[Th. 2.1]{peron2024uniform}: 
\begin{thm}
\label{2T0}
Let $\sigma_{\pm}>0$, and  $\mu_{+}>0$. If the interface $\Sigma$ is Lipschitz, there are constants $\mu_{\star}$ and $C>0$ such that for all $\mu_{r}\geqslant\mu_{\star}$, Maxwell problem \eqref{MS} with boundary conditions \eqref{PIbc} and data $\jj\in\bL^2(\Omega)$ possesses a unique solution $(\EE,\HH)$ in $\bL^2(\Omega)^2$, which satisfies:
\begin{equation}
   \|\HH\|_{0,\Omega} + \|\EE\|_{0,\Omega} 
    + \sqrt{\mu_{r}}\, \|\HH\|_{0,\Omega_-}    \leqslant C \|\jj\|_{0,\Omega} . 
\label{3E1}
\end{equation}
A similar result holds for boundary conditions \eqref{Pcbc}. 
\end{thm}

In that follows, we introduce for convenience the following small parameter in the Maxwell equations \eqref{MS} 
\begin{equation}
\label{epsilon} 
\eps=\dfrac1{\sqrt{\mu_{r}}} \,.
\end{equation}
Then  for $\mu_{r}\geqslant\mu_{\star}$,  we denote by $(\EE_{(\eps)},\HH_{(\eps)})$ the solution of the system \eqref{MS} --\eqref{PIbc}.

\subsection{Electric formulation}
\label{S2.2}

By a standard procedure we deduce from the Maxwell system \eqref{MS}-\eqref{PIbc} the following variational formulation for the electric field $\EE_{(\eps)}$. The variational space is $\bH(\rot,\Omega)$:
\begin{equation}
 \bH(\rot,\Omega) =\{\uu\in\bL^2(\Omega)\; |\ \rot\uu\in\bL^2(\Omega) \}\ , 
\end{equation}
and the variational problem writes
\\[1ex]
{\it Find $\EE_{(\eps)} \in \bH(\rot,\Omega)$ such that for all $\EE^{\prime} \in \bH(\rot,\Omega)$}
\begin{equation}
\int_{\Omega} \big( \frac1{\mur(\eps)}\rot\EE_{(\eps)} \cdot \rot\overline\EE^{\prime}- \kappa_{+}^2
\left(1+ i \frac{\bs}{\omega\varepsilon_{0}}\right)\EE_{(\eps)}\cdot\overline\EE^{\prime} \big) \dr\xx = i \omega \mu_{+}\int_{\Omega} \jj  \cdot\overline\EE^{\prime}\,\dr\xx\ ,
\label{FVE0}
\end{equation}
where we have set
\begin{equation}
\label{2Eeps}
  \mur(\eps)={\mathbf{1}_{\Omega\iso}}  
   + {\frac{1}{\eps^2}}\,{\mathbf{1}_{\Omega\con}} 
 \ ,    \quad 
 \quad\mbox{and}\quad    \kappa_{+}:=\omega\sqrt{\varepsilon_{0}\mu_{+}} \ .
\end{equation}	
\Bk

\section{Multiscale expansion for the electric field}
\label{S3}

 In that follows we assume that the right-hand side $\jj$ in \eqref{MS} is smooth, and its support does not meet the domain  $\Omega_-$. We assume  also that Assumption \ref{H2} holds: \Bk
\begin{hyp}
\label{H2}
We assume that the surfaces $\Sigma$ (interface) and $\Gamma$ (external boundary) are smooth.
\end{hyp}
Let $\cU_-$ be a tubular neighborhood of the surface $\Sigma$ in the conductor part $\Omega_-$, see Figure~\ref{Fig1}. We denote by $(y_{\alpha},y_{3})$ a local {\em normal coordinate system} to the surface $\Sigma$ in  $\cU_-$: Here, $y_\alpha$, $\alpha=1,2$, are tangential coordinates on $\Sigma$ and $y_3$ is the normal coordinate to $\Sigma$, cf.\ \cite{CDFP11}.
\input contribF4T.tex
\begin{figure}[ht]
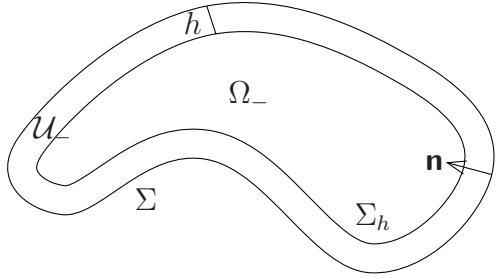

\begin{center}
\def\EpsTube{-10}
\figinit{1.1pt}
\pssetupdate{yes}
\figpt 8:(-48,0)
\figpt 22:(-92,15.5)\figpt 23:(-27.2,60.8)
\figpt 24:(46.3,41.5)\figpt 25:(70.8,2.8)
\figpt 26:(29,-31)\figpt 27:(-27.2,8)\figpt 28:(-77,-11)
\figpt 2:(-24,51)
\figpt 3:(50,50)
\figpt 16:(-70,20)\figpt 15:(0,30)
\figpt 17:(40,10) \figpt 18:(-60,-20)
\figpt 21:(50,65) \figpt 31:(30,-18)  \figpt 34:(-32,55) \figpt 36:(-200,30)
 \figpt 35:(-80,18) 
\figpt 19:(61,65) \figpt 20:(3,27)  \figpt 30:(42.6,33)  \figpt 32:(55,7) \figpt 33:(61.6,6.1) 
\psbeginfig{}
\pssetfillmode{no}\pssetgray{0}
\pscurve[22,23,24,25,26,27,28,22,23,24]
\figptscontrolcurve 40,\NbC[22,23,24,25,26,27,28,22,23,24]
\psEpsLayer \EpsTube,\NbC[40,41,42,43,44,45,46,47,48,49,50,51,52,53,54,55,56,57,58,59,60,61]
\psline[23,2]
\psarrow[25,32]
\pssetfillmode{yes}\pssetgray{0.5}
\psendfig
\figvisu{\figBoxA}{}
{
\figwritew 15: $\Omega\con$(6pt)
\figwritec [34]{$h$}
\figwrites 8: $\Sigma$ (2pt)
\figwriten 31: $\Sigma_{h}$ (2pt)
\figwritew 32: $\nn$(1pt)
\figwritec [35]{$\cU_{-}$}
\figsetmark{$\figBullet$}
}
\centerline{\box\figBoxA}
 \caption{A tubular neighbourhood of the surface $\Sigma$}
\label{Fig1}	
\end{center}
\end{figure}

Let $\sigma_{\pm}>0$, $\mu_{+}>0$. In that follows, in the framework of Section \ref{S2.1},  the electric field  $\EE_{(\eps)}$ is denoted by $\EE^{+}_{(\eps)}$ in the non-magnetic part $\Omega_{+}$, and by $\EE^{-}_{(\eps)}$ in the magnetic conducting part $\Omega_{-}$. Then both parts possess series expansions in powers of $\eps$: 
\begin{gather}
\label{6E4a}
   \EE^{+}_{(\eps)}(\xx)  \approx\sum_{j\geqslant0} \eps^j\EE^{+}_j (\xx)  
     \, ,\quad \xx\in\Omega_+ \, ,
\\
\label{6E4b}
 \EE^{-}_{(\eps)}(\xx) \approx\sum_{j\geqslant0} \eps^j\EE^{-}_j (\xx; \eps) 
 \, ,\quad\xx\in\Omega_- \, ,
  \\
  \label{6E4c}
   \quad\mbox{with}\quad   
   \EE^-_j(\xx;\eps) = \chi(y_3) \,\WW_j(y_\beta,\frac{y_3}{\eps})\, .
\end{gather}
In \eqref{6E4c}, the function $\yy\mapsto\chi(y_3)$ is a smooth cut-off with support in $\overline\cU_-$, and it is such that  $\chi= 1$ in a smaller tubular neighborhood of $\Sigma$. The vector fields 
$\WW_{j}: (y_\alpha,Y_3)\mapsto\WW_{j}(y_\alpha,Y_3)$ are \textit{profiles} defined on $\Sigma\times\R^+$: They are exponentially decreasing with respect to $Y_{3}$ and are smooth in all variables.

In this section we provide the first terms of a multiscale expansion for the electric field at large relative magnetic permeability, cf. Subsection \ref{S3.1}. Then we prove error estimates for this expansion, cf. Subsection \ref{Svalid-bis}.  For the sake of completeness we give elements of proofs for the multiscale expansion in Section \ref{AppA}.

\subsection{First terms of the multiscale expansion}
\label{S3.1}

In this section we provide the construction of the first profiles $\WW_{j}=(\cW_{j},\fke_{j})$ and of the first terms $\EE^+_j$. The normal component $\fke_0$ of the first profile in the magnetic conductor is zero:
\begin{equation*}
\fke_{0}=0\, .
\end{equation*}
Then, the first term of the electric field in the non-magnetic region solves Maxwell equations with perfectly insulating electric boundary conditions on $\partial\Omega_{+}= \Sigma \cup\Gamma$ (we recall $\kappa_{+}=\omega\sqrt{\varepsilon_{0}\mu_{+}}$, cf. 
\eqref{2Eeps}):
\begin{equation}
\label{E0p}
 \left\{ 
   \begin{array}{lll}
    \rot\rot  \EE^+_{0} - \kappa_{+}^2({1+\frac{i}{\delta_{+}^2}}) \EE^+_{0} =  i \omega\mu_{+}\jj   \quad&\mbox{in}\quad \Omega_{+}
\\[0.5ex]
  \EE^+_{0} \cdot\nn= 0 \quad\mbox{and}\quad\!  \rot \EE^+_{0}\times\nn = 0 \quad &\mbox{on}\quad \Sigma \cup\Gamma \ ,
   \end{array}
    \right.
\end{equation}
with $\delta_{+} =\sqrt{{\omega\varepsilon_0} /{\sigma_{+}}}$. Thus the trace $\bsE_0$ of $\EE^+_0$ on the interface $\Sigma$ is \emph{tangential}. 

The first profile in the ferromagnetic region is a tangential field which is exponential with a complex rate  $\lambda$:
\begin{equation}
\label{W0cd}
\cW_{0}(y_{\beta},Y_{3}) = \bsE_{0}(y_{\beta})\;\mathrm{e}^{-\lambda Y_{3}}   
\, .
\end{equation}
Here the scalar $\lambda$ is given by: 
\begin{equation}
\label{lambda}
\lambda=\kappa_{+} \sqrt[4]{1+ \frac1{\delta_{-}^{4}}}\,\mathrm{e}^{\di i\tfrac{\theta(\delta_{-})-\pi}{2}} \quad \mbox{with}\quad  \delta_{-} =\sqrt{{\omega\varepsilon_0} /{\sigma_{-}}}  \ ,\quad \mbox{and}\quad  \theta(\delta_{-})=\arctan\frac1{\delta_{-}^{2}} \, .
\end{equation} 
Note that $\Re \lambda>0$. Note also that, if $\bsE_0$ is not identically $0$, there exist  constants $c, C>0$ independent of $\eps$ such that
\begin{equation}
\label{Eestim}
   c\sqrt \eps \leqslant
   \|\EE^-_0(\,\cdot\,;\eps) \|_{0,\Omega_{-}}\leqslant C\sqrt \eps \ .
\end{equation}    

 The next term which is determined in the asymptotics is the normal component $\fke_1$ of the profile $\WW_1$:
\begin{equation}
\label{fke1a}
\fke_{1}(y_{\beta},Y_{3}) =\lambda^{-1}  \Div_{\Sigma}  \bsE_{0}(y_{\beta}) \;\mathrm{e}^{-\lambda Y_{3}}\,.
\end{equation}
Here the operator $\Div_{\Sigma}$ is the surface divergence operator on $\Sigma$. 
The next term in the  non-magnetic region solves:
\begin{equation}
\label{E1p}
 \left\{
   \begin{array}{lll}
    \rot\rot \EE^+_1 - \kappa_{+}^2({1+\frac{i}{\delta_{+}^2}}) \EE^+_{1}  = 0  \quad&\mbox{in}\quad \Omega_{+}
\\[0.5ex]
   \rot\EE^+_1\times\nn= - \lambda \bsE_0 \quad &\mbox{on}\quad \Sigma
\\[0.5ex]
   \EE^+_{1} \cdot\nn= 0 \quad\mbox{and}\quad\!  \rot \EE^+_{1}\times\nn = 0   
   \quad &\mbox{on}\quad \Gamma.
   \end{array}
    \right.
\end{equation}
 Like above, $\bsE_1$ is the trace of $\EE^+_1$ on the interface $\Sigma$.
Then the tangential components of the profile $\WW_{1}$ are given by the tangential field  $\cW_{1}$:
\begin{equation}
\label{W1cd}
  {\cW_{1}}(y_{\alpha},Y_{3}) =
   \Big[ \bsE_{1}+Y_{3} \big(\cH
   -\cC\big) \bsE_{0}\Big](y_{\alpha}) \; \mathrm{e}^{-\lambda Y_{3}}\, .
\end{equation}    
Here $\cH=\tfrac12\, b_{\alpha}^{\alpha}$ is the \textit{mean curvature} of the surface $\Sigma$\footnote{The sign of $\cH$ depends on the orientation of the surface $\Sigma$. As a convention, the unit normal vector $\nn$ on the surface $\Sigma$ is inwardly oriented to $\Omega_{-}$, see Figure~\ref{Fig1}. }, and $\cC$ is the curvature tensor field on  $\Sigma$ defined by 
\begin{equation}
\label{EcC}
(\cC  \bsE)_{\alpha}=b^{\beta}_{\alpha} \sE_{\beta} \, ,
\end{equation}
with $b_{\alpha}^{\beta}=a^{\beta\gamma}b_{\gamma\alpha}$, and $a^{\beta\gamma}$ is the inverse of the metric tensor $a_{\beta\gamma}$ in $\Sigma$, and $b_{\gamma\alpha}$ is the curvature tensor in $\Sigma$.
\Bk

 The next term in the non-magnetic region solves:
\begin{equation}
\label{E2p}
 \left\{
   \begin{array}{lll}
    \rot\rot \EE^+_2 - \kappa_{+}^2 ({1+\frac{i}{\delta_{+}^2}}) \EE^+_{2}  = 0  \quad&\mbox{in}\quad \Omega_{+}
\\[0.8ex]
   \rot\EE^+_2\times\nn=     -\lambda (\nn \times\bsE_{1})\times\nn + \left(\cH-\cC \right) \bsE_{0}
    \quad &\mbox{on}\quad \Sigma
\\[0.8ex]
   \EE^+_{2} \cdot\nn= 0 \quad\mbox{and}\quad\!  \rot \EE^+_{2}\times\nn = 0   
   \quad &\mbox{on}\quad \Gamma.
   \end{array}
    \right.
\end{equation}

\subsection{Validation of the multiscale expansion}
\label{Svalid-bis}

\newcommand  {\CC}{\boldsymbol{\mathsf C}}
\newcommand  {\jjw}{\widetilde{\boldsymbol{\mathsf\jmath}}}

The validation of the multiscale expansion \eqref{6E4a}-\eqref{6E4b} for the electric field $\EE_{(\eps)}$ 
consist in proving estimates for remainders $\RR_{m;\,\eps}$  which are defined as 
\begin{equation}
\label{6E5} 
  \RR_{m;\,\eps}   = \EE_{(\eps)} - \sum_{j=0}^m \eps^j\EE_j \quad\mbox{in}\quad\Omega\,.
\end{equation} 
This is done by an evaluation of the right hand side when the Maxwell operator is applied to $\RR_{m;\,\eps}$. By construction (cf. Sections \ref{AE1}, \ref{A2}, and \ref{A3}), we obtain (recall that $\mur(\eps)$ is defined by \eqref{2Eeps})
\begin{equation}
\label{6E6}
\left\{
   \begin{array}{lll}
   \rot \rot\RR^+_{m;\,\eps} -  \alpha_+ \kappa_{+}^2\RR^+_{m;\,\eps} &=\quad 0 
   \quad & \mbox{in}\quad \Omega_+
   \\[0.5ex]
\eps^{2}  \rot \rot\RR^-_{m;\,\eps} - \alpha_- \kappa_{+}^2\RR^-_{m;\,\eps} &=\quad \jj^-_{m;\,\eps} 
   \quad & \mbox{in}\quad \Omega_-
   \\[0.5ex]
   \big[\RR_{m;\,\eps}\times\nn\big]_\Sigma &=\quad 0 \quad  &\mbox{on}\quad \Sigma
   \\[0.5ex]
   \big[\mur(\eps)^{-1}\rot\RR_{m;\,\eps}\times\nn\big]_\Sigma &=\quad \g_{m;\,\eps}  
   \quad  &\mbox{on}\quad \Sigma
   \\[0.5ex]
   \RR^+_{m;\,\eps}\cdot\nn=0  \, \, \mbox{  and  }  \rot\RR^+_{m;\,\eps}\times\nn &=\quad 0 \quad &\mbox{on}\quad \partial\Omega\,.
   \end{array}
    \right.
\end{equation}
Here, $\underline{\alpha}=(\alpha_{+}, \alpha_{-})$, $\alpha_+=1+i/\delta_{+}^2$ and $\alpha_-=1+i/\delta_{-}^2$, and $[\uu\times\nn]_\Sigma$ denotes the jump of $\uu\times\nn$ across $\Sigma$. The right hand sides (residues) $\jj^-_{m;\,\eps}$ and $\g_{m;\,\eps}$ are, roughly, of the order $\eps^{m}$ : 
for all $m\in\N$, we have $\jj^-_{m;\,\eps}=\mathcal{O}({\eps^{m+1}})$ and $\g_{m;\, \eps}=\mathcal{O}({\eps^{m}})$, and we have  the following estimates 
\begin{equation}
\label{6E7}
   \|\jj^-_{m;\,\eps}\|_{0,\Omega_-}   \leqslant C_m\eps^{m+1}     \quad \mbox{and} \quad 
   \|\g_{m;\,\eps}\|_{\frac12,\Sigma}   
\leqslant
   C_m\eps^{m},
\end{equation}
where $C_m>0$ is independent of $\eps$. The main result of this section is the following.

\begin{thm}
\label{6T1}
In the framework above,  there exists $\eps_{1}>0$ such that  for all $m\in\N$ and $\eps\in(0,\eps_1]$, the remainder $\RR_{m;\,\eps}$ \eqref{6E5} satisfies the optimal estimate
\begin{equation}
\label{6E8}
   \|\RR^+_{m;\,\eps}\|_{0,\Omega_+} + \|\rot\RR^+_{m;\,\eps}\|_{0,\Omega_+} + 
   \eps^{-\frac12} \|\RR^-_{m;\,\eps}\|_{0,\Omega_-} + 
    \eps\|\rot\RR^-_{m;\,\eps}\|_{0,\Omega_-} \leqslant C_m\eps^{m+1}.
   \!\!\!\!
\end{equation}
\end{thm}
\begin{proof}
Our proof is based on uniform estimates for the electric and magnetic fields  $(\EE_{(\eps)},  \HH_{(\eps)})$ \eqref{3E1},  cf. Theorem \ref{2T0}.

\noindent{\sc Step 1.}  Since $\g_{m;\,\eps} \neq 0$, according to \eqref{6E6} the vector field $\mur(\eps)^{-1}\rot\RR_{m;\,\eps}$ does not define an element of the Hilbert space $\bH(\rot,\Omega)$:
\begin{equation*}
\bH(\rot,\Omega) =\{\uu\in\bL^2(\Omega)\; |\ \rot\uu\in\bL^2(\Omega) \}\ .
\end{equation*}
Hence  we can not apply Theorem \ref{2T0}  directly to the remainder $\RR_{m;\,\eps}$  because the vector field $\rot\mur(\eps)^{-1}\rot\RR_{m;\,\eps} -  \underline{\alpha} \kappa_{+}^2\RR_{m;\,\eps}$ does not define an element of $\bL^2(\Omega)$. 
To overcome this difficulty  we are going to introduce a corrector $\CC_{m;\,\eps}$ satisfying suitable estimates and so that
\begin{equation}
\label{6E6C}
\left\{
   \begin{array}{lll}
   \big[(\RR_{m;\,\eps} - \CC_{m;\,\eps})\times\nn\big]_\Sigma 
   &=\quad 0 \quad  &\mbox{on}\quad \Sigma
   \\[0.6ex]
   \big[\mur(\eps)^{-1}\rot(\RR_{m;\,\eps}-\CC_{m;\,\eps})\times\nn\big]_\Sigma 
   &=\quad 0   \quad  &\mbox{on}\quad \Sigma
    \\[0.6ex]
       (\RR_{m;\,\eps}-\CC_{m;\,\eps})\cdot\nn  =0  \, \, \mbox{  and  }  \rot (\RR_{m;\,\eps}-\CC_{m;\,\eps}) \times\nn  &=\quad 0   \quad  &\mbox{on}\quad \partial\Omega\ .
   \end{array}
    \right.
\end{equation}

\noindent
Construction of $\CC_{m;\,\eps}$: We take $\CC_{m;\,\eps}=0$ in $\Omega_-$ and use a trace lifting to define $\CC_{m;\,\eps}$ in $\Omega_+$. According to \eqref{6E6}, and since $\mur(\eps)=1$ in $\Omega_+ $, it suffices that
\begin{equation}
\label{6E7C}
\left\{
   \begin{array}{lll}
   \CC^+_{m;\,\eps} \times\nn 
   &=\quad 0 \quad  &\mbox{on}\quad \Sigma
   \\[0.5ex]
 \rot\CC^+_{m;\,\eps} \times\nn 
   &=\quad \g_{m;\,\eps}   \quad  &\mbox{on}\quad \Sigma
     \\[0.5ex]
   \CC^+_{m;\,\eps} \cdot\nn   =0  \, \, \mbox{  and  }  \rot \CC^+_{m;\,\eps} \times\nn  &=\quad  0   \quad  &\mbox{on}\quad \partial\Omega \ .
   \end{array}
    \right.
\end{equation}
Then, denoting by $C_\beta$ and $C_3$ the tangential and normal components of $\CC^+_{m;\,\eps}$ associated with a system of normal coordinates $\yy=(y_\beta,y_3)$, and by $g_\beta$ the components of $\g_{m;\,\eps}$ the above system becomes 
\begin{equation}
\label{6E8C}
\left\{
   \begin{array}{lll}
   C_\beta 
   &=\quad 0 \quad  &\mbox{on}\quad \Sigma
   \\[0.5ex]
  \partial_3 C_\beta - \partial_\beta C_3 
   &=\quad g_\beta   \quad  &\mbox{on}\quad \Sigma
    \\[0.5ex]
C_3 =0     \, \, \mbox{  and  }    \partial_3 C_\beta - \partial_\beta C_3 =0 \quad  &\mbox{on}\quad \partial\Omega \ .
   \end{array}
    \right.
\end{equation}
Hence the problem \eqref{6E8C} can be solved in $\bH^2(\Omega_+)$ choosing $C_3=0$ and a standard lifting of the first two traces on $\Sigma$ and $\partial\Omega$ with the estimate
\begin{equation}
\label{6E9C}
   \|\CC^+_{m;\,\eps}\|_{2,\Omega_+} \leqslant C\|\g_{m;\,\eps}\|_{\frac12,\Sigma} \ .
\end{equation}

We deduce from the estimate \eqref{6E7}, and \eqref{6E9C}
\begin{equation}
\label{6E11}
   \|\CC^+_{m;\,\eps}\|_{2,\Omega_+}
 \leqslant 
   C\eps^{m} ,
\end{equation}
where $C$ may depend on $m$.
We set
\begin{equation}
\label{6E12}
   \widetilde\RR_{m;\,\eps} := \RR_{m;\,\eps} - \CC_{m;\,\eps} 
   \quad\mbox{and}\quad
  \jjw_{m;\,\eps} := 
  \rot\mur(\eps)^{-1}\rot\widetilde\RR_{m;\,\eps} - \kappa_{+}^2 \underline{\alpha}\widetilde\RR_{m;\,\eps} \ .
\end{equation}
Hence by construction, we have $\jjw_{m;\,\eps}\in\bL^{2}(\Omega)$ with the estimates
\begin{equation}
\label{6E13}
   \|\jjw_{m;\,\eps}\|_{0,\Omega}  \leqslant 
   C\eps^{m} \ .
\end{equation}

\noindent{\sc Step 2.}
We can apply Theorem \ref{2T0}
  to the couple $(\EE,\HH) = (\widetilde\RR_{m;\,\eps},(i\omega\bm)^{-1}\rot\widetilde\RR_{m;\,\eps})$ and, thanks to \eqref{6E12}, we obtain 
\begin{equation*}
   \|\widetilde\RR_{m;\,\eps}\|_{0,\Omega} +
    \|\bm^{-1}\rot\widetilde\RR_{m;\,\eps}\|_{0,\Omega}
   \leqslant C \|\jjw_{m;\,\eps}\|_{0,\Omega}\,.
\end{equation*}
Combining this estimate with \eqref{6E11} and \eqref{6E13}, we deduce
\begin{equation}
\label{6E14}
   \|\RR_{m;\,\eps}\|_{0,\Omega} + \|\bm^{-1}\rot\RR_{m;\,\eps}\|_{0,\Omega}
   \leqslant  C\eps^{m},
\end{equation}
where $C$ may depend on $m$.

\noindent{\sc Step 3.}
In order to have an optimal estimate for $\RR_{m;\,\eps}$, we use \eqref{6E14} for $m+2$. We obtain (recall that $\bm=\mu_{+}$ in $\Omega_{+}$,   $\bm=\mu_{-}$ in $\Omega_{-}$, and $\mu_{-}=\eps^{-2}\mu_{+}$)
\begin{equation}
\label{6E15}
   \|\RR_{m+2;\,\eps}\|_{0,\Omega} +  \frac{1}{\mu_{+}} \|\rot\RR_{m+2;\,\eps}^{+}\|_{0,\Omega_{+}} + \frac{\eps^{2}}{\mu_{+}} \|\rot\RR_{m+2;\,\eps}^{-}\|_{0,\Omega_{-}} 
   \leqslant  C\eps^{m+2} \, .
\end{equation}
But we have the formula
\begin{equation}
\label{6E16}
   \RR_{m;\,\eps} = \sum_{j=m+1}^{m+2} \eps^j \EE_j + \RR_{m+2;\,\eps} \, ,
\end{equation}
and using  \eqref{EAb}, we have also for any $j\in\N$
\begin{equation}
\label{6E17}
   \|\EE^+_j\|_{\bH(\rot,\Omega_+)} + 
   \eps^{-\frac12} \|\EE^-_j\|_{0,\Omega_-} + 
   \eps^{\frac12} \|\rot\EE^-_j\|_{0,\Omega_-} \leqslant C.
\end{equation}
We finally deduce the estimate \eqref{6E8} from \eqref{6E15} to \eqref{6E17}.
\end{proof}

\subsection{Alternative proof of error estimates in the non-magnetic part $\Omega_{+}$}
\label{Svalid}

In the previous subsection, we validated the asymptotic expansion \eqref{6E4a}-\eqref{6E4b} for the electric field $\EE_{(\eps)}$ 
by proving estimates for remainders $\RR_{m;\,\eps}^{\EE}$ which are defined as  (cf. \eqref {6E5})
\begin{equation}
\label{E6E5}
   \RR_{m;\,\eps}^{\EE} = \EE_{(\eps)} - \sum_{j=0}^m \eps^j\EE_j \quad\mbox{in}\quad\Omega\, .
\end{equation}
In this subsection, we prove alternatively error estimates in the non-magnetic part $\Omega_{+}$ by using error estimates for the magnetic field \cite[Th. 4.2]{peron2024uniform}.

Using Maxwell equations \eqref{MS}, we have 
\begin{equation}
\label{EARH+}
 \RR_{m;\,\eps}^{\EE,+}= (\sigma_{+} -i \omega\varepsilon_0 )^{-1} \rot \RR_{m;\,\eps}^{\HH,+}  \quad \mbox{and} \quad \rot\RR_{m;\,\eps}^{\EE,+}=i \omega\mu_{+}\RR_{m;\,\eps}^{\HH,+} \quad\mbox{in}\quad\Omega_{+} \, .
\end{equation}
Here $\RR_{m;\,\eps}^{\HH}$ are remainders which are defined from  the asymptotic  expansion \eqref{6E4a}-\eqref{6E4b} \Bk for the magnetic field $\HH_{(\eps)}$ 
\begin{equation}
\label{E6E6}
   \RR_{m;\,\eps}^{\HH} = \HH_{(\eps)} - \sum_{j=0}^m \eps^j\HH_j \quad\mbox{in}\quad\Omega\,.
\end{equation}
Then, Theorem 4.2 in Ref. \cite{peron2024uniform} yields that there exists $\eps_{1}>0$ such that  for all $m\in\N$, and $\eps\in(0,\eps_{1})$, we have the following optimal estimate 
\begin{multline}
\label{6E8}
   \|\RR^{\HH,+}_{m;\,\eps}\|_{0,\Omega_+} + \|\rot\RR^{\HH,+}_{m;\,\eps}\|_{0,\Omega_+} + 
   \eps^{-\frac12} \|\RR^{\HH,-}_{m;\,\eps}\|_{0,\Omega_-} + 
   \eps^{\frac12} \|\rot\RR^{\HH,-}_{m;\,\eps}\|_{0,\Omega_-} \leqslant C_m\eps^{m+1}.
   \!\!\!\!
\end{multline}
Hence, according to \eqref{EARH+}, we retrieve estimates for remainders $\RR_{m;\,\eps}^{\EE}$ in the non-magnetic part $\Omega_{+}$ 
\begin{equation*}
   \|\RR_{m;\,\eps}^{\EE,+} \|_{0,\Omega_+} + \|\rot\RR_{m;\,\eps}^{\EE,+} \|_{0,\Omega_+}  \leqslant c_m\eps^{m+1}, 
\end{equation*}
compare with Eq. \eqref{6E8} in Theorem \ref{6T1}.

\begin{rem} 
The presence of profiles in the multiscale expansions for the electric field and for the magnetic field  in the magnetic part  $\Omega_{-}$ prevents to link the remainders $\RR_{m;\,\eps}^{\EE,-}$ and $\RR_{m;\,\eps}^{\HH,-}$ via the Maxwell equations in a similar way as \eqref{EARH+}.  Nevertheless, we have the following estimates for remainders $\RR_{m;\,\eps}^{\EE,-}$ (cf. Eq. \eqref{6E8} in Theorem \ref{6T1}) :
\begin{equation*} 
   \eps^{-\frac12} \|\RR_{m;\,\eps}^{\EE,-} \|_{0,\Omega_-} + 
\eps    \|\rot\RR_{m;\, \eps}^{\EE,-} \|_{0,\Omega_-} \leqslant C'_m  \eps^{m+1}.
\end{equation*}
\end{rem}

\section{A measure of the magnetic skin effect at large relative permeability}
\label{S3.2}

In this section  we introduce a suitable skin depth function to measure the boundary layer phenomenon,  cf. Section \ref{S4.1}.  As an application of the multiscale expansion given in Section \ref{S3}  for the electric field, we obtain the asymptotic behavior of this function at large relative permeability. Then, we compare this result with the literature on this topic \cite{DFP10,CDFP11,PePo21},  cf. Section \ref{S4.2}.

\subsection{Definition and asymptotic behavior of a characteristic length}
\label{S4.1}
For a data $\jj$, let us define  
$$
\WW_{(\eps)}(y_\alpha,y_3) := \EE^-_{(\eps)}(\xx),\quad 
y_\alpha\in\Sigma,\quad 0\leqslant y_3<h_0\ ,
$$
for $h_{0}$ small enough. Hereafter for any  $1$-form field $\WW=(\cW,\fke)$, $|\WW |$ denotes the modulus of $\WW$ which is defined from the following identity
\begin{equation*}
 |\WW |^{2}=   a^{\alpha\beta}(y_{3}){\cW_{\alpha}}\overline{\cW_{\beta}}+ |\fke|^{2}, 
\end{equation*}
 where $a^{\alpha\beta}(h)$ is the inverse of the metric tensor of the manifold $\Sigma_{h}$ (cf. \eqref{Einvmet}, Sec. \ref{AppA}),  and $\langle\cdot\,,\cdot\rangle$ denotes the corresponding scalar product. 

\subsubsection*{Definition of a characteristic length}
\begin{defn}
\label{DefL}
Let $\Sigma$ be a smooth surface, and $\jj$ a data of problem \eqref{MS} such that for all $y_{\alpha}$ in $\Sigma$, $\WW_{(\eps)}(y_{\alpha},0) \neq 0$. The skin depth is the length $\cL(\mu_{-},\sigma_{-},y_{\alpha})$ defined on $\Sigma$ and taking the smallest positive value such that
\begin{equation}
\label{epaisspeau}
   |\WW_{(\eps)}\big(y_{\alpha},\cL(\mu_{-},\sigma_{-},y_{\alpha})\big) |=
   |\WW_{(\eps)}(y_{\alpha},0) | \ {\mathrm{e}}^{-1}\,.
\end{equation}
\end{defn}

Thus the length $\cL(\mu_{-},\sigma_{-},y_{\alpha})$ corresponds to  a distance from the interface where the modulus $| \WW_{\eps} |$ has decreased  by a factor $\mathrm{e}$. This length depends in particular on the magnetic permeability $\mu_{-}$, on the conductivity $\sigma_{-}$, and on each point $y_{\alpha}$ in the interface $\Sigma$. It depends also  on the data $\jj$.

\subsubsection*{Asymptotic behavior of the characteristic length function when $\mu_{r}\to \infty$}

To measure the magnetic skin effect with the characteristic function, we introduce the classical {\em skin depth} parameter which is given by 
\begin{equation}
 \label{El}
 \ell(\mu_{-},\sigma_{-})=\sqrt{\frac{2}{\omega\mu_{-}\sigma_{-}}}\ ,
\end{equation}
and we introduce a scalar function $\phi$  defined  as 
\begin{equation}
\label{Ephi}
\phi(\delta)=\dfrac{1}{\sqrt{2}} \dfrac1{\sqrt[4]{1+ \delta^{4}}}\,\left(\sin\left(\frac12 \arctan\frac1{\delta^{2}}\right) \right)^{-1} \, ,
\end{equation}
for any $\delta>0$. Note  that we have in particular:   $\phi(\delta)\to 1$ as $\delta\to0$ and  $\phi(\delta)\sim \sqrt2 \delta$ as $\delta\to\infty$,  see also Figure~\ref{phi} for the graph of function $\phi$.

As a consequence of \eqref{W0cd} and \eqref{W1cd}, we obtain 
\begin{equation}
\label{EW}
\left\{ \begin{array}{l}
   |\WW_{(\eps)}(y_{\alpha},y_{3}) |^2 = 
   |\bsE_{0}(y_{\alpha}) |^2 \ \gm(y_{\alpha},y_3;\eps) \ 
   \mathrm{e}^{-2 y_{3} \Re(\lambda) / \eps} \ , 
   \quad\mbox{with}\\[1ex]
   \gm(y_{\alpha},y_3;\eps) :=  \\ \displaystyle \qquad\qquad
   1 + 2 y_{3} \cH(y_{\alpha}) +
   2 \eps \ \frac{\Re \big\langle  \bsE_{0}(y_{\alpha}) , \bsE_{1}(y_{\alpha}) \big\rangle}
   { | \bsE_{0}(y_{\alpha}) |^{2}} 
   + \cO\big((\eps+y_{3})^2\big)
    \, . 
\end{array}\right.
\end{equation}
We prove this formula \eqref{EW} in Section \ref{S4.3} within the proof of the next result, Theorem \ref{T1}. 
 Relying on this formula, one can exhibit the asymptotic behavior   at large relative permeability  of the skin depth $\cL(\mu_{-},\sigma_{-},y_{\alpha})$ that depends in particular both on the product $\ell(\mu_{-},\sigma_{-})\phi(\sqrt{{\omega\varepsilon_0} /{\sigma_{-}}})$, and  on  the mean curvature $\cH$ of the surface $\Sigma$ (the boundary of the domain $\Omega_{-}$): 
\begin{thm}
\label{T1}
Recall that $\ell(\mu_{-},\sigma_{-})$ is defined by \eqref{El}, $\delta_{-} =\sqrt{{\omega\varepsilon_0} /{\sigma_{-}}}$,  and $\phi$  is the  function  defined by \eqref{Ephi}. Let $\mu_{+}>0$. 
Let $\Sigma$ be a regular surface with mean curvature $\cH$. 
We assume that $\bsE_{0}(y_{\alpha}) \neq 0$. Then the skin depth $\cL(\mu_{-},\sigma_{-},y_{\alpha})$ has the following behavior for large relative permeability:
\begin{equation}
\label{daepaisspeau}
   \cL(\mu_{-},\sigma_{-},y_{\alpha})=
   \ell(\mu_{-},\sigma_{-}) \phi(\delta_{-}) \Big(1+\cH(y_{\alpha})\, \ell(\mu_{-},\sigma_{-})  \phi(\delta_{-}) + \cO(\mu_{r}^{-1})\Big),
   \quad
   \mu_{r}\to \infty\, .
\end{equation}
\end{thm}
The proof of this result is postponed to Section \ref{S4.3}.
\begin{rem}
The higher order terms $\cO(\mu_{r}^{-1})$ in Eq. \eqref{daepaisspeau} do depend on the data $\jj$ of problem \eqref{MS}.
\end{rem}

\subsection{Comparisons with the literature}
\label{S4.2}

\subsubsection*{Comparison with the eddy-current problem}

A preliminary study  was done in \cite{PePo21} for the magnetic potential in a 2D eddy-current problem. For this interface problem where $\Sigma$ is a smooth and closed curve,  the skin depth has the following asymptotic behavior for large relative permeability, cf. \ \cite[Prop.\,2.10]{PePo21}:
 \begin{equation}
\label{daepaisspeau2D}
   \cL(\mu_{-},\sigma_{-}; \xi)\sim
\ell(\mu_{-},\sigma_{-})\Big(1+\frac{\kappa(\xi)}{2}\, \ell(\mu_{-},\sigma_{-})    \Big),
   \quad
   \mu_{r}\to \infty\, ,
\end{equation}
where $\kappa(\xi)$ is a scalar curvature at the point $\XX(\xi)\in\Sigma$ and $\XX(\xi)$ denotes an arc-length parameterization of $\Sigma$.  
We retrieve this equivalent fonction $\ell(\mu_{-},\sigma_{-})\Big(1+\frac{\kappa(\xi)}{2}\, \ell(\mu_{-},\sigma_{-})\Big)$  defined on $\Sigma$  from formula \eqref{daepaisspeau} only in the limit case $\delta_{-}\to 0$ since then $\phi(\delta_{-})\to 1$ (see also Figure~\ref{phi}), and  by replacing the mean curvature $\cH(y_{\alpha})$  by $\frac{\kappa(\xi)}{2}$.  
\Bk 

\subsubsection*{Comparison with the asymptotic behavior of the skin depth at  high conductivity in non-magnetic materials}

Since we have both $\delta_{-}\to 0$ when $\sigma_{-} \to \infty$, and  $\phi(\delta)\to 1$  as  $\delta\to 0$,
 \begin{figure}[ht]
\begin{center}
\includegraphics[keepaspectratio=true,width=10cm]{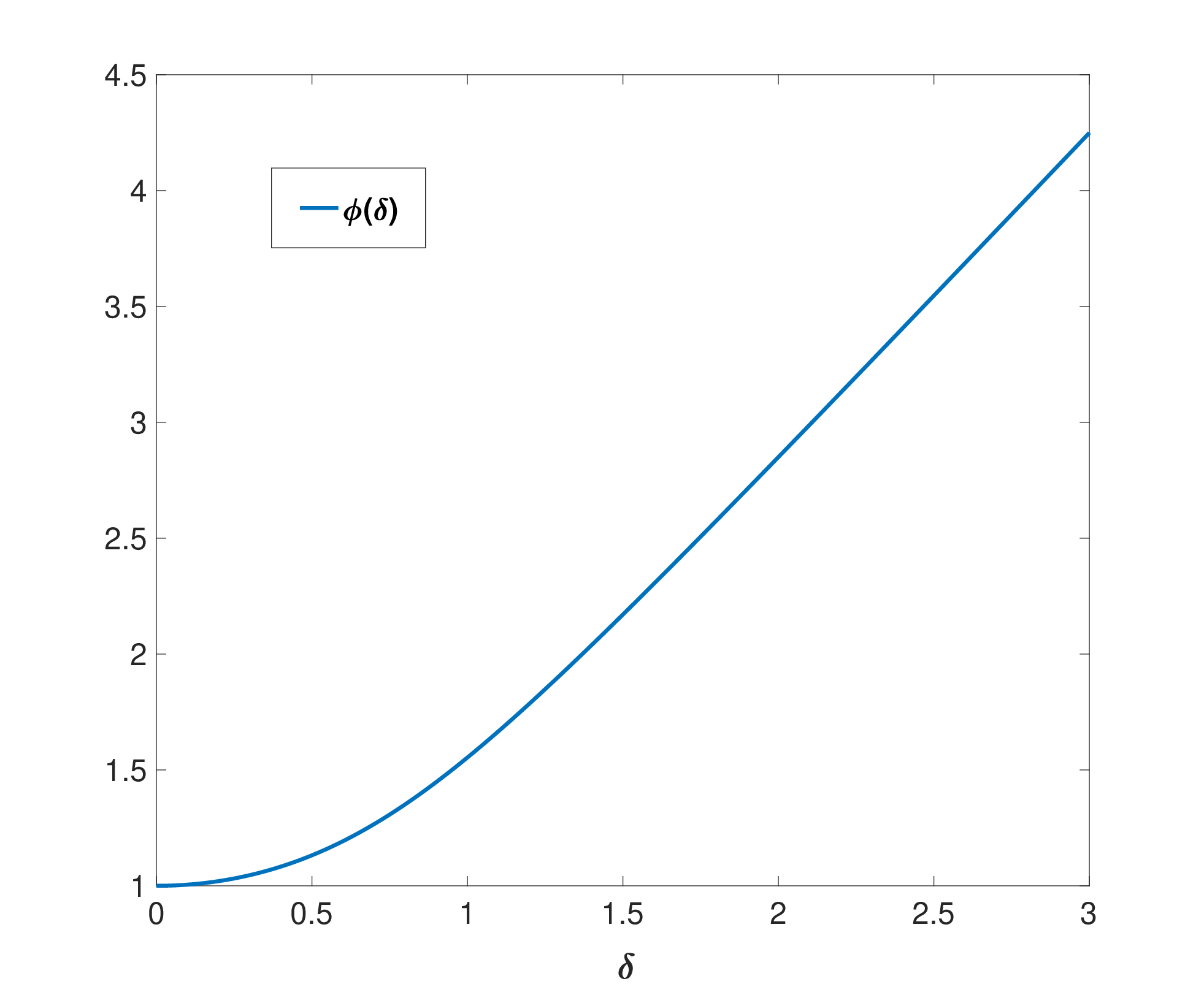}
\caption{Graph of $\phi$ with respect to $\delta$}
\label{phi}
\end{center} 
\end{figure}
we obtain
  \begin{equation*}
 \ell(\mu_{-},\sigma_{-}) \phi(\delta_{-}) \Big(1+\cH(y_{\alpha})\, \ell(\mu_{-},\sigma_{-})  \phi(\delta_{-}) \Big)
  \sim
\ell(\mu_{-},\sigma_{-}) \Big(1+\cH(y_{\alpha})\, \ell(\mu_{-},\sigma_{-}) \Big) \ ,    \quad  
   \sigma_{-}\to \infty\, .
\end{equation*}
This equivalent fonction $\ell(\mu_{-},\sigma_{-}) \Big(1+\cH(y_{\alpha})\, \ell(\mu_{-},\sigma_{-}) \Big)$ defined on $\Sigma$ generalizes the following  equivalent fonction given by the asymptotic behavior of the skin depth at high conductivity in non-magnetic materials, cf.\ \cite[Th.\,3.2]{CDFP11}:
\begin{equation}
\label{daepaisspeau-old}
   \cL(\sigma_{-},y_{\alpha})\sim  \ell(\mu_{0},\sigma_{-})\Big(1+\cH(y_{\alpha})\,  \ell(\mu_{0},\sigma_{-}) 
   \Big),
   \quad
   \sigma\to \infty\, ,
\end{equation}
where  $\ell(\mu_{0},\sigma_{-})$ is the classical skin depth \eqref{El} when  the magnetic permeability $\mu_{-}$ is the  vacuum magnetic permeability  $\mu_{0}$.
\begin{rem}[Low frequency limit]
For the same reason, since we have also $\delta_{-}\to 0$ as $\omega\to 0$, we obtain the same equivalent fonction  at low frequency and at high conductivity $\sigma_{-}$: 
  \begin{equation*}
 \ell(\mu_{-},\sigma_{-}) \phi(\delta_{-}) \Big(1+\cH(y_{\alpha})\, \ell(\mu_{-},\sigma_{-})  \phi(\delta_{-}) \Big)
  \sim
\ell(\mu_{-},\sigma_{-}) \Big(1+\cH(y_{\alpha})\, \ell(\mu_{-},\sigma_{-}) \Big) \ ,    \quad 
   \omega\to 0\, .
\end{equation*}
\end{rem}

\subsubsection*{Influence of the geometry of the interface on the magnetic skin effect}

Formula \eqref{El} shows particularly that  the larger the mean curvature of the interface $\Sigma$,  the  larger the characteristic length, for large values of $\mu_{r}$. It is a common point with the 3D Maxwell problem in non-magnetic materials (cf. \eqref{daepaisspeau-old}). 
It is also a similarity  with the 2D eddy-current model in which the larger the scalar curvature $\kappa(\xi)$ of the interface,  the  larger the characteristic length $ \cL(\mu_{-},\sigma_{-}; \xi)$, for large values of $\mu_{r}$  (cf. formula \eqref{daepaisspeau2D}).

\subsection{Proof of Theorem \ref{T1}}
\label{S4.3}

We adapt the proof which is given in \cite{DFP10} to obtain the asymptotic behavior of the skin depth at high conductivity in a non-magnetic material.  
\begin{proof}
Let $\WW(\eps)$ be the vector  field defined as $\WW(\eps)(y_{\alpha},Y_{3}) := \WW_{(\eps)}(y_{\alpha},y_{3})$ with $y_{3}=\eps Y_{3}$. The proof of \eqref{EW} is based on calculi in covariant components of $\WW(\eps)$ which can be seen also as a $1$-form field.  Using \eqref{6E4b}-\eqref{6E4c}, we have:
$$
\WW(\eps)=\big({\cW_{0,\alpha}},0\big)+\eps \big(\cW_{1,\alpha}, \fke_{1} \big)+\cO(\eps^2)\,.
$$
By definition of the modulus of a $1$-form field,  we obtain:
\[
   | \WW(\eps) |^2 =
   a^{\alpha\beta}(y_{3}){\cW_{0,\alpha}}\overline{\cW_{0,\beta}}
   + 2 \eps\, a^{\alpha\beta}(y_{3}) 
   \Re {\cW_{0,\alpha}}\overline{{\cW_{1,\beta}}} 
   + \cO(\eps^2)\,.
\]
Then using the inverse of the metric tensor (cf. Eq. \eqref{Einvmet}, Section \ref{AppA}),  and using \eqref{W0cd}-\eqref{W1cd}, we  obtain the following identities successively:
\[
\begin{split}
   | \WW(\eps) |^2 &=
   (a^{\alpha\beta}
   + 2 \eps Y_{3} b^{\alpha\beta}) \cW_{0,\alpha} \overline{\cW_{0,\beta}} 
   + 2 \eps\, a^{\alpha\beta} \Re \cW_{0,\alpha} \overline{\cW_{1,\beta}}
   + \cO\big((\eps+y_{3})^2\big) \\ 
   &= \Big[ | \bsE_{0} |^2 
   + 2\eps Y_3 b^{\alpha\beta} \sE_{0,\alpha} \overline{\sE_{0,\beta}}
   + 2 \eps\, a^{\alpha\beta} \Re \sE_{0,\alpha}
   \Big( \overline{\sE_{1,\beta}}+Y_{3} \big(\cH\,\overline{\sE_{0,\beta}}
   - b_{\beta}^{\sigma}\overline{\sE_{0,\sigma}}\big)\Big)\\
 &+ \cO\big((\eps+y_{3})^2\big) 
   \Big] \,\mathrm{e}^{-2\Re(\lambda) Y_{3}}   \\ 
        &= \Big[ | \bsE_{0} |^2   + 2 \eps\, a^{\alpha\beta} \Re \sE_{0,\alpha}
   \Big( \overline{\sE_{1,\beta}}+Y_{3} \cH\,\overline{\sE_{0,\beta}}\Big)+ \cO\big((\eps+y_{3})^2\big) 
   \Big] \,\mathrm{e}^{-2\Re(\lambda) Y_{3}} \\
   &= \Big[ 
   | \bsE_{0} |^2 + 2 \eps\Re \langle  \bsE_{0}, \bsE_{1}\rangle
   + 2 \eps Y_{3} \cH\, | \bsE_{0} |^2 
   + \cO\big((\eps+y_{3})^2\big)
   \Big]  \,\mathrm{e}^{-2\Re(\lambda) Y_{3}}\, .
\end{split}
\]
Then,  since $\bsE_{0}(y_{\alpha}) \neq 0$, the following identity enables to define $\gm(y_{\alpha},y_3;\eps)$ : 
$$
   |\WW_{(\eps)}(y_{\alpha},y_{3}) |^2 =  
   |\bsE_{0}(y_{\alpha}) |^2 \gm(y_{\alpha},y_3;\eps) \, .
$$
From the last identities we infer: 
\begin{multline*} 
   \gm(y_{\alpha},y_3;\eps)  = \\
   \Big[1+2 \eps  | \bsE_{0}(y_{\alpha}) |^{-2} 
   \Re \big\langle  \bsE_{0}(y_{\alpha}) , \bsE_{1}(y_{\alpha}) \big\rangle 
   +2 y_{3} \cH(y_{\alpha}) + \cO\big((\eps+y_{3})^2\big)
   \Big]  \,\mathrm{e}^{-2 y_{3} \Re(\lambda) / \eps} \,,
\end{multline*}
and 
\[
 \gm(y_{\alpha},0;\eps) = 
   1+2 \eps  | \bsE_{0}(y_{\alpha}) |^{-2} 
   \Re \big\langle  \bsE_{0}(y_{\alpha}) , \bsE_{1}(y_{\alpha}) \big\rangle 
   + \cO\big(\eps^2\big) \,.
\]
The definition of $\cL(\mu_{-},\sigma_{-},y_{\alpha})$ leads to the following identity:
$$
\gm\big(y_{\alpha},\cL(\mu_{-},\sigma_{-},y_{\alpha});\eps\big) / \gm(y_{\alpha},0;\eps)=\mathrm{e}^{-2}\,.
$$
Finally, we obtain Eq. \eqref{daepaisspeau}  from the last identity by doing a Taylor expansion in $\eps$, and by using the following identity (recall that  we have $\eps= \sqrt{\mu_{+}/ \mu_{-}}$, and $\lambda$ is defined  by Eq. \eqref{lambda}): 
\begin{equation*}
\eps /  \Re(\lambda) =  \ell(\mu_{-},\sigma_{-}) \phi(\delta_{-})\,. 
\end{equation*}
\end{proof}

\section{Impedance boundary conditions}
\label{S5}

As a by-product of the multiscale expansion for the electric field (cf Eqs. \eqref{6E4a}-\eqref{6E4b}) it is possible to derive impedance boundary conditions (IBCs) on the surface $\Sigma$   for the electromagnetic field. In \cite{ECCOMAS24} we developed briefly the asymptotic method for  deriving these IBCs. In this work we give elements of proof for the derivation of IBCs (Sec. \ref{IBC}), and for a uniform stability and convergence result (Sec. \ref{SValid}).

\subsection{Derivation of impedance boundary conditions}
\label{IBC}

The derivation of  IBCs is based also on a regular expansion for the magnetic field in the non-magnetic part. The magnetic field possesses the following series expansion in powers of $\eps$ in $\Omega_{+}$, cf. \cite{peron2024uniform}: 
\begin{equation}
\label{6E4aH}
   \HH^{+}_{(\eps)}(\xx)  \approx\sum_{j\geqslant0} \eps^j\HH^{+}_j (\xx)  
       \, ,\quad \xx\in\Omega_+ \, ,
\end{equation}
The terms $\HH^{+}_{j}$ are defined in \cite{peron2024uniform}. 

For all $j\in\N$, each term $\HH^{+}_{j}$ can also be defined directly from the term $\EE^{+}_{j}$ (in \eqref{6E4a}) by using the Faraday's law in $\Omega_{+}$ which write (cf \eqref{MS}): 
\begin{equation}
    \rot \ \EE^{+}_{j} - i \omega\mu_{+} \HH^{+}_{j}  = 0     \quad\mbox{in}\quad\Omega_{+}\, .
\label{EFLj}
\end{equation}
In that follows, we use this identity \eqref{EFLj} to derive IBCs  up to the third order of approximation with respect to $\eps$, cf. Eqs. \eqref{EIBC1}-\eqref{EIBC3}-\eqref{EIBC4b}. We deduce a family of simpler problems (cf. \cite{ECCOMAS24})
 \begin{equation}
  \label{E1}
   \left\{
    \begin{array}{lllll}
     \rot  \EE_{k}^{\eps} - i \omega\mu_{+} \HH_{k}^{\eps} &= &0    \quad &\mbox{in}\quad \Omega_{+}
            \\[8pt]
          \rot \HH_{k}^{\eps}  + (i\omega\varepsilon_0 - \sigma_{+}) \EE_{k}^{\eps} & = &  \jj  \quad &\mbox{in}\quad \Omega_{+}
       \\[8pt]
      \HH_{k}^{\eps}\times\nn & =
      & \D_{k,\eps}\left( (\nn\times \EE_{k}^{\eps}) \times \nn\right) & \mbox{on}\quad {\Sigma}\ 
      \\[6pt]
         \EE_{k}^{\eps} \cdot\nn &= & 0 \quad&\mbox{on}
                 \quad \partial\Omega
            \\[6pt]
            \HH_{k}^{\eps}\times\nn &=  &0\quad &\mbox{on}         \quad \partial\Omega
         \ ,
    \end{array}
  \right.
\end{equation}
where $\D_{k,\eps}$ is a surface differential operator defined  as  
 \begin{gather}
\label{D0}
\D_{0,\eps}=0 
  \ ,
\\
\label{D1}
\D_{1,\eps}= \dfrac1{\sqrt{\mu_{-}}} \sqrt[4]{\varepsilon^{2}_0+ \left(\frac{\sigma_{-}}{\omega}\right)^{2}} \,\mathrm{e}^{\di \tfrac{i}{2} \arctan\left(\frac{\sigma_{-}}{\omega\varepsilon_{0}}\right)} \Id \ ,
\\
\label{D2}
\D_{2,\eps}= \left(\dfrac1{\sqrt{\mu_{-}}} \sqrt[4]{\varepsilon^{2}_0+ \left(\frac{\sigma_{-}}{\omega}\right)^{2}} \,\mathrm{e}^{\di \tfrac{i}{2} \arctan\left(\frac{\sigma_{-}}{\omega\varepsilon_{0}}\right)}
 +\dfrac1{i\omega{\mu_{-}}} \left(\cH-\cC \right)
\right) \Id
\, , 
 \end{gather}
for $k=0,1,2$,  respectively.

\subsubsection*{First order impedance boundary condition}
Define $\HH^{+}_{0}$  from $\EE^{+}_{0}$ by using \eqref{EFLj} for $j=0$: 
\begin{equation}
    \rot \ \EE^{+}_{0} - i \omega\mu_{+} \HH^{+}_{0}  = 0     \quad\mbox{in}\quad\Omega_{+}\, .
\label{EFL0}
\end{equation}
Then  we obtain the perfectly insulating electric boundary conditions on $\Sigma$ by using \eqref{E0p}:
\begin{equation}
\label{EIBC1}
 \EE^+_{0} \cdot\nn=0 \quad\mbox{and}\quad
 \HH^{+}_{0}
 \times\nn = 0  \quad \mbox{on}\quad \Sigma \ .
\end{equation}

\subsubsection*{Second order impedance boundary condition}
Using \eqref{E0p}-\eqref{E1p}, we obtain: 
\begin{equation*}
  \rot (\EE^+_{0} +\eps  \EE^+_{1}) \times\nn = -\eps \lambda \bsE_0 \quad \mbox{on}\quad \Sigma \ .
\end{equation*}
Define $\EE^+_{1,\eps} = \EE^+_{0} + \eps  \EE^+_{1}$. Hence we have
\begin{equation}
\label{EBC1}
  \rot \EE^+_{1,\eps} \times\nn = -\eps \lambda (\nn\times\EE^+_{1,\eps})\times\nn  +\g_{1,\eps}  \quad \mbox{on}\quad \Sigma \ . 
\end{equation}
Here, $\g_{1,\eps}=\eps^{2} \lambda  (\nn\times\EE^+_{1})\times\nn$.  
Then define $\HH^{+}_{1}$  from $\EE^{+}_{1}$ by using \eqref{EFLj} for $j=1$: 
\begin{equation}
    \rot \ \EE^{+}_{1} - i \omega\mu_{+} \HH^{+}_{1}  = 0     \quad\mbox{in}\quad\Omega_{+}\, .
\label{EFL1}
\end{equation}
Setting $\HH^+_{1,\eps} = \HH^+_{0} + \eps  \HH^+_{1}$, we infer 
\begin{equation}
\label{EBC1bis}
 i \omega\mu_{+} \HH^+_{1,\eps} \times\nn = -\eps \lambda (\nn\times\EE^+_{1,\eps})\times\nn  
 +\g_{1,\eps} 
 \quad \mbox{on}\quad \Sigma \ . 
\end{equation}
Then,  by neglecting the term of order $\eps^{2}$ in \eqref{EBC1bis}, we deduce the IBC \begin{equation}
\label{EIBC2}
 i \omega\mu_{+} \HH^\eps_{1}  \times\nn = -\eps \lambda (\nn\times \EE^\eps_{1})\times\nn  \quad \mbox{on}\quad \Sigma \ . 
\end{equation}
Finally we obtain (we recall: $\lambda=\kappa_{+} \sqrt[4]{1+ \frac1{\delta_{-}^{4}}}\,\mathrm{e}^{\di i\tfrac{\theta(\delta_{-})-\pi}{2}}$, $\delta_{-} =\sqrt{{\omega\varepsilon_0} /{\sigma_{-}}}$, and $\theta(\delta_{-})=\arctan\frac1{\delta_{-}^{2}}$, cf. Eq. \eqref{lambda}, and 
$\kappa_{+}=\omega\sqrt{\varepsilon_{0}\mu_{+}}$, cf.  
 Eq. \eqref{2Eeps}) 
 \begin{equation}
\label{EIBC3}
\HH^\eps_{1}  \times\nn = \dfrac1{\sqrt{\mu_{-}}} \sqrt[4]{\varepsilon^{2}_0+ \left(\frac{\sigma_{-}}{\omega}\right)^{2}} \,\mathrm{e}^{\di \tfrac{i}{2} \arctan\left(\frac{\sigma_{-}}{\omega\varepsilon_{0}}\right)} (\nn\times \EE^\eps_{1})\times\nn  \quad \mbox{on}\quad \Sigma \ . 
\end{equation}
\begin{rem}[Comparison with the Leontovich condition for strongly absorbing material]
\label{Leon}
This impedance boundary condition \eqref{EIBC3} is different from the classical Leontovich condition for strongly absorbing material (large conductivity limit), cf. {\it e.g.} \cite{Mo03,HJN08}. However we retrieve at least formally the Leontovich condition from \eqref{EIBC3} as a first approximation when $\delta_{-} =\sqrt{{\omega\varepsilon_0} /{\sigma_{-}}}\to 0$. Indeed, from the impedance factor in \eqref{EIBC3},  straightforward calculi  based on Taylor expansions  as $\delta_{-}\to 0$  lead to the Leontovich condition: 
 \begin{equation}
\label{ELBC}
(\nn\times \EE)\times\nn = \sqrt\frac{\mu_{-}\omega}{\sigma_{-}}\,\mathrm{e}^{\di -\tfrac{i\pi}{4}}\HH  \times\nn  \quad \mbox{on}\quad \Sigma \ .
\end{equation}
\end{rem}

\subsubsection{Third order impedance boundary condition}

In the same way we derive a third order IBC with respect to the small parameter $\eps$ by collecting the first terms $\EE^+_{2,\eps} = \EE^+_{0} + \eps  \EE^+_{1} +  \eps^{2}  \EE^+_{2}$ and by using \eqref{E0p}-\eqref{E1p}-\eqref{E2p}. By construction we obtain
\begin{equation}
\label{EIBC4a}
  \rot \EE^+_{2,\eps} \times\nn = -\eps\big(\lambda (\nn\times \EE^+_{2})\times\nn -\eps\left(\cH -\cC\right)  (\nn\times\EE^+_{2,\eps})\times\nn\big)+\g_{2,\eps}    \quad \mbox{on}\quad \Sigma \ . 
\end{equation}
Here, $\g_{2,\eps}=\eps^{3}\big(\lambda +\eps\left(\cH -\cC\right) \big)   (\nn\times \EE^+_{2})\times\nn -\eps^{3}\left(\cH -\cC\right)  (\nn\times\EE^+_{1})\times\nn$.  
Then define $\HH^+_{2}$ from $\EE^+_{2}$  by using \eqref{EFLj} for $j=2$: 
\begin{equation}
    \rot  \EE^+_{2}  - i \omega\mu_{+} \HH^+_{2}  = 0     \quad\mbox{in}\quad\Omega_{+}\, .
\label{EFL2}
\end{equation} 
Setting $\HH^+_{2,\eps} = \HH^+_{0} + \eps  \HH^+_{1}+ \eps^{2}  \HH^+_{2}$, we infer 
\begin{equation}
\label{EIBC4abis}
 i \omega\mu_{+} \HH^+_{2,\eps} \times\nn = -\eps\big(\lambda (\nn\times \EE^+_{2,\eps})\times\nn -\eps\left(\cH -\cC\right)  (\nn\times\EE^+_{2,\eps})\times\nn\big) + \g_{2,\eps} 
   \quad \mbox{on}\quad \Sigma \ . 
\end{equation}
Then by neglecting the term 
$\g_{2,\eps}=\mathcal{O}(\eps^{3})$ in \eqref{EIBC4abis} we identify the  third order IBC:
\begin{equation}
\label{EIBC4b}
\HH^\eps_{2} 
 \times\nn = \left(\dfrac1{\sqrt{\mu_{-}}} \sqrt[4]{\varepsilon^{2}_0+ \left(\frac{\sigma_{-}}{\omega}\right)^{2}} \,\mathrm{e}^{\di \tfrac{i}{2} \arctan\left(\frac{\sigma_{-}}{\omega\varepsilon_{0}}\right)}
 +\dfrac1{i\omega{\mu_{-}}} \left(\cH-\cC \right)
\right)  (\nn\times \EE^\eps_{2}
 )\times\nn  \quad \mbox{on}\quad \Sigma \ . 
\end{equation}

\subsection{Error estimates for the impedance boundary conditions} 
\label{SValid}

In this section we prove the following  convergence result  on approximate models  with  impedance boundary conditions  \eqref{E1}. 
\begin{thm}
\label{TValid}
In the framework above (cf. the beginning of Sec. \ref{S3}), there exists $\eps_{\ast}>0$ such that for all $\eps\in (0, \eps_{\ast})$,   for all $k \in \{0,1,2\}$, problem   \eqref{E1}  
  has a unique solution $(\EE_{k}^{\eps}, \HH_{k}^{\eps})$ in $(\bL^{2}(\Omega_{+}))^{2}$,  
which satisfies:
\begin{equation}
\label{errork}
\| \EE_{(\eps)}-\EE_{k}^{\eps} \|_{0,\Omega_{+}}
+\|\HH_{(\eps)} -\HH_{k}^{\eps} \|_{0,\Omega_{+}}
 \leqslant C_k {\eps^{k+1}} 
 \ ,
\end{equation}
with a constant $C_k$ independent of $\eps$.  
\end{thm}
\begin{proof}
Since $\sigma_{+}>0$, problem  \eqref{E1} possesses a unique solution $(\EE_{k}^{\eps}, \HH_{k}^{\eps})$ in $(\bL^{2}(\Omega_{+}))^{2}$, \Bk and there is a constant $C>0$ such that we have the following estimate
\begin{equation*}
\|\EE_{k}^{\eps} \|_{0,\Omega_+}+\| \HH_{k}^{\eps} \|_{0,\Omega_+}
 \leqslant C \| \jj\|_{0,\Omega_+} \, , 
\end{equation*}
with a constant $C$ independent of $\eps$, cf. {\it e.g.} \cite{Mo03}.  

Let us introduce  the truncated expansions $(\EE^+_{k, \eps},\HH^+_{k, \eps})$  in the non-magnetic part 
 \begin{equation*}
\EE^+_{k, \eps}:={\EE^+_0 + \eps \EE^+_1 + {{\eps^2}}  \EE^+_2+ \cdots+ {{\eps^k}}  \EE^+_k } \quad\mbox{in}\quad \Omega_{+} \ , 
\end{equation*}
and 
\begin{equation*}
\HH^+_{k, \eps}:={\HH^+_0 + \eps	\HH^+_1 + {{\eps^2}}  \HH^+_2+ \cdots+ {{\eps^k}}  \HH^+_k }  
\quad\mbox{in}\quad \Omega_{+} \ .
\end{equation*}
Then the  convergence result \eqref{errork} can be obtained as a consequence of estimates for remainders $(\RR_{k;\,\eps}^{\EE,+},\RR_{k;\,\eps}^{\HH,+})$ (cf. Section  \ref{S3})    together with a uniform stability result,  and by using the truncated expansions  $(\EE^+_{k, \eps},\HH^+_{k, \eps})$  as intermediate quantities. 

By construction of the asymptotic expansions,  we have  for $k=1,2$, (cf. \eqref{EBC1bis}-\eqref{EIBC4abis}, Section \ref{IBC})
\begin{equation}
  \label{E2}
   \left\{
    \begin{array}{lllll}
     \rot  \EE_{k, \eps} - i \omega\mu_{+} \HH_{k, \eps} &= &0    \quad &\mbox{in}\quad \Omega_{+}
       \\[8pt]
          \rot \HH_{k, \eps}  + (i\omega\varepsilon_0 - \sigma_{+}) \EE_{k, \eps}& = &  \jj  \quad &\mbox{in}\quad \Omega_{+}
       \\[8pt]
      \HH_{k, \eps}\times\nn & =
      & \D_{k,\eps}\left( (\nn\times \EE_{k, \eps}) \times \nn\right) +( i \omega\mu_{+})^{-1}\g_{k,\eps} 
      & \mbox{on}\quad {\Sigma}\ 
      \\[6pt]
         \EE_{k, \eps} \cdot\nn &= & 0 \quad&\mbox{on}         \quad \partial\Omega
                     \\[6pt]
            \HH_{k, \eps}\times\nn &=  &0\quad &\mbox{on}         \quad \partial\Omega
         \ ,
    \end{array}
  \right.
\end{equation}
Here the impedance operator $\D_{k,\eps}$ is defined  by \eqref{D1}, and \eqref{D2}, for $k=1,2$,  respectively. For $k=0$, the third equation in \eqref{E2} has to be replaced by the boundary condition \eqref{EIBC1}. 
Then define 
\begin{equation*}
  \QQ_{k;\,\eps}^{\EE,+} := \EE^+_{k, \eps} -\EE_{k}^{\eps} \quad \mbox{and} \quad   \QQ_{k;\,\eps}^{\HH,+} := \HH^+_{k, \eps} -\HH_{k}^{\eps} \, . 
\end{equation*}
By using Eqs. \eqref{E1} and \eqref{E2}, for  $k=1,2$, we infer 
\begin{equation}
  \label{E3}
   \left\{
    \begin{array}{lllll}
     \rot   \QQ_{k;\,\eps}^{\EE,+}  - i \omega\mu_{+}  \QQ_{k;\,\eps}^{\HH,+} &= &0 
   \quad &\mbox{in}\quad \Omega_{+}
       \\[8pt]
          \rot  \QQ_{k;\,\eps}^{\HH,+}  + (i\omega\varepsilon_0 - \sigma_{+}) \QQ_{k;\,\eps}^{\EE,+} & = &  0 \quad &\mbox{in}\quad \Omega_{+}
                 \\[8pt]
      \QQ_{k;\,\eps}^{\HH,+} \times\nn & =
      & \D_{k,\eps}\left( (\nn\times  \QQ_{k;\,\eps}^{\EE,+} ) \times \nn\right) +( i \omega\mu_{+})^{-1}\g_{k,\eps}     
      & \mbox{on}\quad {\Sigma}\ 
      \\[6pt]
         \QQ_{k;\,\eps}^{\EE,+}  \cdot\nn &= & 0 \quad&\mbox{on}         \quad \partial\Omega
                     \\[6pt]
             \QQ_{k;\,\eps}^{\HH,+} \times\nn &=  &0\quad &\mbox{on}         \quad \partial\Omega
         \ .
    \end{array}
  \right.
\end{equation}
 For $k=0$, we simply have $(  \QQ_{0;\,\eps}^{\EE,+}, \QQ_{0;\,\eps}^{\HH,+})=(0,0)$. 
Since $\sigma_{+}>0$, problem  \eqref{E3} possesses a unique solution $(\QQ_{k;\,\eps}^{\EE,+}, \QQ_{k;\,\eps}^{\HH,+})$ in $(\bL^{2}(\Omega_{+}))^{2}$, and since  $\| \g_{k,\eps}\|_{0,\Sigma} = \mathcal{O}({\eps^{k+1}})$, we have the following estimate 
\begin{equation*}
\| \QQ_{k;\,\eps}^{\EE,+} \|_{0,\Omega_{+}}
+\| \QQ_{k;\,\eps}^{\HH,+}  \|_{0,\Omega_{+}}
 \leqslant C_k {\eps^{k+1}} 
 \ .
\end{equation*}
 with a constant $C_{k}$ independent of $\eps$. 
 
 Finally by using this estimate together with estimates for remainders $(\RR_{k;\,\eps}^{\EE,+},\RR_{k;\,\eps}^{\HH,+})$ (cf. Thm. \ref{6T1} and Eq. \eqref{6E8}  in Section  \ref{S3}), we deduce the convergence result \eqref{errork} from the following identities
\begin{equation*}
 \EE^+_{(\eps)}-\EE_{k}^{\eps} = (\EE^+_{(\eps)}-\EE^+_{k, \eps}) + (\EE^+_{k, \eps} -\EE_{k}^{\eps})  = \RR_{k;\,\eps}^{\EE,+} +  \QQ_{k;\,\eps}^{\EE,+} \, , 
\end{equation*}
 and 
 \begin{equation*}
 \HH^+_{(\eps)}-\HH_{k}^{\eps} = (\HH^+_{(\eps)}-\HH^+_{k, \eps}) + (\HH^+_{k, \eps} -\HH_{k}^{\eps})  = \RR_{k;\,\eps}^{\HH,+} +  \QQ_{k;\,\eps}^{\HH,+}  \, . 
\end{equation*}
\Bk
\end{proof}

\section{Elements of proof for the multiscale expansion}
\label{AppA}

In the framework of Section \ref{S3} (cf. the beginning of Sec. \ref{S3}), we have the following result (recall in particular that $(y_{\alpha},y_{3})$ is a local {\em normal coordinate system} defined in  $\cU_-$,  a tubular neighborhood of the surface $\Sigma$, cf. Fig.~\ref{Fig1} in Sec. \ref{S3}):
\begin{thm}
\label{ThA1}
In the framework of Section \ref{S3}, the electric field $\EE_{(\eps)}$ possesses the asymptotic expansion:  
\begin{gather}
\label{EAa}
   \EE^+_{(\eps)}(\xx) \approx \sum_{j\geqslant0} \eps^j\EE^+_j(\xx)  \, ,
\\  
\label{EAb}
   \EE^-_{(\eps)}(\xx) \approx \sum_{j\geqslant0} \eps^j\EE^-_j(\xx;\eps) 
   \quad\mbox{with}\quad   
   \EE^-_j(\xx;\eps) = \chi(y_3) \,\WW_j(y_\alpha,\frac{y_3}{\eps})\, ,
\end{gather}
 such that 
\begin{equation*}
\WW_j(y_\alpha,Y_{3})\longrightarrow 0   \quad\mbox{as}\quad  
Y_{3}\longrightarrow \infty \, .
\end{equation*}
In \eqref{EAb} the function $\yy\mapsto\chi(y_3)$ is a smooth cut-off with support in $\overline\cU_-$ and equal to $1$ in a smaller tubular neighborhood of $\Sigma$, cf. Fig.~\ref{Fig1} in Section \ref{S3}.   

Moreover, for any $j\in\N$, we have 
\begin{equation}
\label{EAd}
    \EE^+_j\in\bH(\rot,\Omega_+)\quad\mbox{and}\quad \WW_j\in\bH(\rot,\Sigma\times\R_+).
\end{equation}
\end{thm}

Hereafter, we present elements of proof of this theorem and details about the terms in asymptotics \eqref{EAa}-\eqref{EAb}. In \S\ref{AE1}, we expand the ``electric'' Maxwell operators in power series of $\eps$ inside the boundary layer $\cU_-$. We deduce in \S\ref{A2} the equations satisfied by the electric profiles, and we derive explicitly the first ones in \S\ref{A3}.

\subsection{Expansion of the operators}
\label{AE1}
Integrating by parts in the electric variational formulation \eqref{FVE0}, we find the following Maxwell transmission problem
for the electric field (we recall $ \kappa_{+}=\omega\sqrt{\varepsilon_{0}\mu_{+}}$, cf. \eqref{2Eeps})
\begin{equation}
\label{EME}
 \left\{
   \begin{array}{lll}
   \rot\rot  \EE^+_{(\eps)} - \kappa_{+}^2({1+\frac{i}{\delta_{+}^2}}) \EE^+_{(\eps)} =i \omega\mu_{+}\jj   \quad&\mbox{in}\quad \Omega_{+}
\\[0.5ex]
\eps^{2}\rot\rot  \EE^-_{(\eps)} -\kappa_{+}^2({1+\frac{i}{\delta_{-}^2}})  \EE^-_{(\eps)} =0
 \quad&\mbox{in}\quad \Omega_{-}
\\[0.5ex]
  \rot \EE^+_{(\eps)}\times\nn= \eps^{2} \rot \EE^-_{(\eps)}\times\nn \quad &\mbox{on} \quad \Sigma
\\[0.5ex]
 \EE^+_{(\eps)}\times\nn= \EE^-_{(\eps)}\times\nn  \quad &\mbox{on}\quad \Sigma
\\[0.5ex]
  \EE^+_{(\eps)} \cdot\nn= 0 \quad\mbox{and}\quad\!  \rot \EE^+_{(\eps)}\times\nn = 0  \quad &\mbox{on}\quad \Gamma .
   \end{array}
    \right.
\end{equation}
Here $\delta_{-} = \sqrt{{\omega\varepsilon_0} /{\sigma_{-}}}$ and $\delta_{+} = \sqrt{{\omega\varepsilon_0} /{\sigma_{+}}}$.  
Since $\jj\in\bH(\Div,\Omega)$, we have  also
 $\Div (i\omega\varepsilon_0 - \bs) \EE_{(\eps)}  \in\L^2(\Omega)$. Hence we deduce  the following extra transmission conditions 
 \begin{equation}
\label{AEn}
 ({1+\frac{i}{\delta_{+}^2}})\EE^+_{(\eps)}\cdot\nn=({1+\frac{i}{\delta_{-}^2}})\EE^-_{(\eps)}\cdot\nn  \quad \mbox{on}\quad \Sigma \ .
\end{equation}

We denote by $\LL(y_{\alpha}, h;D_{\alpha}, \partial_{3}^h)$ the 2d order Maxwell operator $\eps^{2}\rot\rot-\kappa_{+}^2({1+\frac{i}{\delta_{-}^2}})\Id$ set in $\cU_{-}$ in a {\em normal coordinate system}. Here $D_{\alpha}$ is the covariant derivative on the interface $\Sigma$, and $\partial_{3}^h$ is  the partial derivative with respect to the normal coordinate $y_{3}=h$. 

Let $a_{\alpha\beta}(h)$ be the metric tensor of the manifold $\Sigma_{h}$, which is the surface contained in $\Omega_-$ at a distance $h$ of $\Sigma$, see Figure~\ref{Fig1}. The metric tensor in such a coordinate system writes
\begin{equation}
\label{Emett}
a_{\alpha\beta}(h)=a_{\alpha\beta}-2 b_{\alpha\beta} h + b_{\alpha}^{\gamma}b_{\gamma\beta} h^2 \, ,
\end{equation}
and its inverse expands in power series of $h$
\begin{equation}
\label{Einvmet}
a^{\alpha\beta}(h)=a^{\alpha\beta}+2 b^{\alpha\beta} h + \mathcal{O}( h^2) \, .
\end{equation}
cf. {\it e.g.} {\rm\cite{DFP10,Fa02}}. With this metric, a three-dimensional vector field $\WW$ can be split into its normal component $\fke$ and its tangential component that can be alternatively viewed as a vector field $\cW^{\alpha}$ or a one-form field $\cW_{\alpha}$ with the relation
\begin{equation}
\label{Emet}
\cW_{\alpha} = a_{\alpha \beta}(h)\cW^{\beta}. 
\end{equation}
Subsequently, we use a property of the covariant derivative, that it acts on functions like the partial derivative:  $D_{\alpha} \fke=\partial_{\alpha} \fke$.

The operator $\LL$ expands in power of $h$ with intrinsic coefficients with respect to $\Sigma$ like the operator  $\rot\rot$, cf. {\it e.g.} \cite{CDFP11}. We make the scaling $Y_{3}=\eps^{-1} h$ to describe the boundary layer with respect to $\eps$. Then, the three-dimensional harmonic Maxwell operator in $\cU_{-}$ is written $\LL[\eps]$. This operator expands in power series of $\eps$ with coefficients intrinsic operators :
\begin{equation*}
 \LL[\eps]=\di\sum_{n=0}^{\infty} \eps^n\LL^{n} 
  \ .
\end{equation*}
We denote by $L_{\alpha}^n$ the surface components of $\LL^{n}$. With the summation convention of repeated two dimensional indices (represented by greek letters), there holds
\begin{equation}
L_{\alpha}^0(\WW)=-\partial_{3}^2 \cW_{\alpha}-\kappa_{+}^2({1+\frac{i}{\delta_{-}^2}}) \cW_{\alpha} 
  \ \  \mbox{and}  \ \ 
L_{\alpha}^1(\WW)=-2b_{\alpha}^\beta\partial_{3} \cW_{\beta} 
+ \partial_{3} D_{\alpha} \fke 
+  b_{\beta}^{\beta}\partial_{3}\cW_{\alpha}\ .
\label{EHsL0L1}
\end{equation}
Here, $\partial_{3}$ is the partial derivative with respect to $Y_{3}$. We denote by $L_{3}^n$ the transverse components of $\LL^{n}$. There holds
\begin{equation}
L_{3}^0(\WW)=-\kappa_{+}^2 ({1+\frac{i}{\delta_{-}^2}})\fke 
  \quad \mbox{and}  \quad
L_{3}^1(\WW)=\gamma_{\alpha}^{\alpha}(\partial_{3}\WW)+  b_{\beta}^{\beta}\partial_{3} \fke\ ,
\label{EHtL0L1}
\end{equation}
where $\gamma_{\alpha\beta}(\VV)=\frac12 (D_{\alpha}\cV_\beta + D_{\beta}\cV_\alpha)-b_{\alpha\beta}\sv$ is the change of metric tensor.

We denote by $\BB(y_{\alpha}, h;D_{\alpha}, \partial_{3}^h)$ the tangent trace operator $\rot\cdot\times\nn$ on $\Sigma$ in a {\em normal coordinate system}. If 
$\WW=(\cW_{\alpha},\fke)$, then
\begin{equation*}
\left( \BB(y_{\alpha}, h;D_{\alpha} , \partial_{3}^h)\WW\right)_{\alpha}=\partial_{3}^h \cW_{\alpha} - D_{\alpha} \fke \ ,
\end{equation*}
see \cite[Sec. A.4]{CDFP11}. We define $\BB[\eps]$ the operator  obtained from $\eps^2\BB$ in $\cU_{-}$ after the scaling $Y_{3}=\eps^{-1} h$. This operator expands in power of $\eps$ :
\begin{equation*}
 \BB[\eps]=\eps\BB^{0}+\eps^2\BB^{1} \ .
\end{equation*}
Thus, denoting by $B_{\alpha}^n$ the surface components of $\BB^{n}$,  we obtain
\begin{equation}
\label{EB0B1}
B_{\alpha}^0(\WW)= \partial_{3}\cW_{\alpha}   \quad \mbox{and}  \quad B_{\alpha}^1(\WW)= -D_{\alpha} \fke \ .
\end{equation}

\subsection{Equations for the coefficients of the electric field}
\label{A2}

According to the second and third equations in system \eqref{EME}, the profiles $\WW_{j}$ and the terms $ \EE_{j}^+$ of the electric field satisfy the following system
\begin{gather}
\label{ELW}
 \LL[\eps] \sum_{j\geqslant0} \eps^j\WW_j(y_{\alpha},Y_{3})=0  \ ,
   \quad \mbox{in }  \quad \Sigma \times I \ ,
   \\
  \label{EBW}
 \sum_{j\geqslant0} \eps^j \rot \EE_{j}^+ \times\nn =  \BB[\eps] \sum_{j\geqslant0} \eps^j\WW_j(y_{\alpha},0) \quad \mbox{on} \quad \Sigma \ .
\end{gather}

We identify in \eqref{ELW}-\eqref{EBW} the terms with the same power in $\eps$. The components of equation \eqref{ELW} are the collections of equations
\begin{equation}
\label{EsLVj}
\LL^0(\WW_{0})=0\ , \quad \LL^0(\WW_{1})+\LL^1(\WW_{0})=0 \ , \quad \mbox{and} \quad \di\sum_{l=0}^{n} \LL^{n-l}(\WW_{l})=0  \ ,
\end{equation}
for all $n\geqslant2$.  The surface components of the equation \eqref{EBW} write 
\begin{multline}
\label{EsBWj}
\rot \EE_{0}^+ \times\nn =0\ ,  \quad \left(\rot \EE_{1}^+ \times\nn \right)_{\alpha} = B_{\alpha}^0(\WW_{0})\ ,
\\
\mbox{and} \quad \left(\rot \EE_{n+2}^+ \times\nn \right)_{\alpha} =B_{\alpha}^0(\WW_{n+1})+B_{\alpha}^1(\WW_{n})   \ ,
\end{multline}
for all $n\geqslant0$.

According to the system \eqref{EME} and \eqref{EB0B1}, the profiles 
$\WW_n = (\cW_{n},\fke_n)$ and the terms $\EE^+_{n}$ have to satisfy, for all $n \geqslant 0$,  
\begin{equation}
\label{EE}
 \left\{
   \begin{array}{clll}
   (i) & \quad &
   -\lambda^2 \fke_n =  \di \sum_{j = 0}^{n-1} L_3^{n-j}( \WW_j) 
 \quad&\mbox{in}\quad \Sigma \times I
     \\[0.9ex]
     (ii) & \quad &
   \rot\rot  \EE^+_{n} - \kappa_{+}^2  ({1+\frac{i}{\delta_{+}^2}})\EE^+_{n} = \delta_n^0 i \omega\mu_{+}\jj   \quad&\mbox{in}\quad \Omega_{+}
\\[0.9ex]
   (iii)  & \quad &  \left( \rot \EE_{n}^+ \times\nn \right)_{\alpha} =   \partial_{3}\cW_{n-1,\alpha}  -D_{\alpha} \fke_{n-2}
  \quad &\mbox{on} \quad \Sigma
  \\[0.9ex]
   (iv) & \quad &
  \EE^+_{n} \cdot\nn= 0 \quad\mbox{and}\quad\!  \rot \EE^+_{n}\times\nn = 0 \quad &\mbox{on}\quad \Gamma 
\\[0.8ex]
   (v) & \quad &
   \partial_{3}^2 \cW_{n,\alpha}-\lambda^2 \cW_{n,\alpha} =  \di \sum_{j = 0}^{n-1} L_\alpha^{n-j}( \WW_j) 
 \quad&\mbox{in}\quad \Sigma \times I
 \\[0.5ex]
    (vi) & \quad &
\cW_{n}= \nn\times \bsE_n\times\nn \quad &\mbox{on}\quad \Sigma \ 
   \end{array}
    \right.
\end{equation}
where $\lambda=\kappa_{+} \sqrt[4]{1+ \frac1{\delta_{-}^{4}}}\,\mathrm{e}^{i\frac{\theta(\delta_{-})-\pi}{2}}$ and $\theta(\delta_{-})=\arctan\frac1{\delta_{-}^{2}}$ (so that  $-\lambda^2=\kappa_{+}^2({1+\frac{i}{\delta_{-}^2}})$ and $\Re \lambda>0$),
 and $\bsE_n$ denotes the trace of $\EE^+_{n}$ on $\Sigma$. In \eqref{EE}, we use the convention that the sums are $0$ when $n=0$. Note that the transmission condition \eqref{AEn}    implies the extra condition
 \begin{equation}
\label{AEnbis}
  ({1+\frac{i}{\delta_{+}^2}}) \EE^+_{n} \cdot\nn= ({1+\frac{i}{\delta_{-}^2}}) \fke_n \quad \mbox{on}\quad \Sigma.
\end{equation} 
 Hereafter,  using the set of equations \eqref{EE}-\eqref{AEn}, we determine the terms $\WW_{n}$ and $\EE^+_{n}$ by induction.

\subsection{First terms of the asymptotics for the electric field}
\label{A3}
According to equation $(i)$ in \eqref{EE}, the normal component $\fke_{0}$ of the first profile in the magnetic conductor vanishes:
\begin{equation}
\label{Efke0}
\fke_{0}=0\ , 
\end{equation}
because $\kappa_{+}\neq 0$, thus $\lambda\neq0$.

Hence, according to equations $(ii)$-$(iv)$ in system \eqref{EE} for $n=0$,  the first asymptotic of the electric field in the dielectric part solves the following problem 
\begin{equation}
\label{E0pbis}
 \left\{ 
   \begin{array}{lll}
    \rot\rot  \EE^+_{0} - \kappa_{+}^2  ({1+\frac{i}{\delta_{+}^2}})\EE^+_{0} =  i \omega\mu_{+}\jj   \quad&\mbox{in}\quad \Omega_{+}
\\[0.8ex]
  \EE^+_{0} \cdot\nn= 0 \quad\mbox{and}\quad\!  \rot \EE^+_{0}\times\nn = 0 \quad &\mbox{on}\quad \Sigma \cup\Gamma .
   \end{array}
    \right.
\end{equation}
Thus the trace $\bsE_0$ of $\EE^+_0$ on the interface $\Sigma$ is \emph{tangential}.

From equations $(v)$-$(vi)$ in system \eqref{EE} for $n=0$, $\cW_{0}$ satisfies the following ODE 
\begin{equation}
\left\{
   \begin{array}{lll}
\partial_{3}^2\cW_{0}(.,Y_{3}) -\lambda^2 \cW_{0}(.,Y_{3})&=0 \quad &\mbox{for} \quad Y_{3}\in I \ ,
\\[0.6ex]
\cW_{0}(.,0)&=(\nn \times\bsE_{0})\times\nn \quad &\mbox{on}\quad \Sigma \ .
\end{array}
  \right.
\label{EW0}
\end{equation}
The unique solution of \eqref{EW0} such that $\cW_{0}\rightarrow 0$ when $Y_{3}\rightarrow\infty$, is the tangential field
\begin{equation*}
\cW_{0}(y_{\beta},Y_{3}) = \bsE_{0}(y_{\beta})\;\mathrm{e}^{-\lambda Y_{3}}\, .
\end{equation*}

Combining with \eqref{Efke0}, we find that the first profile in the conductor region is exponential with the complex rate $\lambda$:
\begin{equation}
\WW_{0}(y_{\beta},Y_{3}) = \bsE_{0} (y_{\beta})\,\mathrm{e}^{-\lambda Y_{3}} \, .
\label{WOcdbis}
\end{equation}

 The next term which is determined in the asymptotics is the normal component $\fke_1$ of the profile $\WW_1$ given by equation $(i)$ of \eqref{EE} for $n = 1$. We obtain 
\begin{equation}
\label{Efke1}
\fke_{1}(y_{\beta},Y_{3}) =\lambda^{-1} D_{\alpha}\bsE^{\alpha}_{0}(y_{\beta}) \;\mathrm{e}^{-\lambda Y_{3}}\,.
\end{equation}
According to \eqref{EE} $(ii)$-$(iv)$, the next term in the dielectric region solves:
\begin{equation}
\label{E1pbis}
 \left\{
   \begin{array}{lll}
    \rot\rot \EE^+_1 - \kappa_{+}^2({1+\frac{i}{\delta_{+}^2}}) \EE^+_{1}  = 0  \quad&\mbox{in}\quad \Omega_{+}
\\[0.8ex]
   \rot\EE^+_1\times\nn= - \lambda \bsE_0 \quad &\mbox{on}\quad \Sigma
\\[0.8ex]
   \EE^+_{1} \cdot\nn= 0 \quad\mbox{and}\quad\!  \rot \EE^+_{1}\times\nn = 0   
   \quad &\mbox{on}\quad \Gamma.
   \end{array}
    \right.
\end{equation}

Recall that $\bsE_1$ is the trace of $\EE^+_1$ on the interface $\Sigma$. We denote by $\sE_{1,\alpha}$ its tangential components. According to equations $(v)$-$(vi)$ in \eqref{EE} for $n = 1$, $\cW_{1}$ satisfies the following ODE (for $Y_3\in I$)
\begin{equation}
\label{EW1}
\left\{
   \begin{array}{lll}
   \partial_{3}^2\cW_{1,\alpha}(.\,,Y_{3}) -\lambda^2 \cW_{1,\alpha}(.\,,Y_{3}) = -2b_{\alpha}^{\sigma}\partial_{3}\cW_{0,\sigma}(.\,,Y_{3})+b_{\beta}^{\beta}\partial_{3}\cW_{0,\alpha}(.\,,Y_{3})
\\[3pt]
\cW_{1,\alpha}(.\,,0)=\sE_{1,\alpha}(.\,,0)\ .
   \end{array}
\right.
\end{equation} 
From \eqref{WOcdbis}, the unique solution of \eqref{EW1} such that $\cW_{1}\rightarrow 0$ when $Y_{3}\rightarrow\infty$ is the following profile 
\begin{equation}
\label{cW1cd}
   {\cW_{1,\alpha}}(y_{\beta},Y_{3}) =
   \Big[ \sE_{1,\alpha}+Y_{3} \big(\cH\,\sE_{0,\alpha}
   -b^{\sigma}_{\alpha} \sE_{0,\sigma}\big)\Big](y_{\beta}) \;\mathrm{e}^{-\lambda Y_{3}},
   \quad\alpha=1,2\,. 
\end{equation}    

 The next term which is determined in the asymptotics is the normal component $\fke_2$ of the profile $\WW_2$ given by equation $(i)$ of \eqref{EE} for $n = 2$. We obtain 
\begin{equation}
\label{Efke2}
\fke_{2}(y_{\beta},Y_{3}) =-\lambda^{-2}\left( L_3^{2}( \WW_0) + L_3^{1}( \WW_1) \right) \, .  
\end{equation}
According to \eqref{EE} $(ii)$-$(iv)$ for $n=2$, the next term in the dielectric region solves the problem:
\begin{equation}
\label{E2pbis}
 \left\{
   \begin{array}{lll}
    \rot\rot \EE^+_2 - \kappa_{+}^2({1+\frac{i}{\delta_{+}^2}}) \EE^+_{2}  = 0  \quad&\mbox{in}\quad \Omega_{+}
\\[0.8ex]
   \left(\rot\EE^+_2\times\nn\right)_{\alpha}=  \partial_{3}\cW_{1,\alpha} 
   -D_{\alpha} \fke_{0}  \quad &\mbox{on}\quad \Sigma
\\[0.8ex]
   \EE^+_{2} \cdot\nn= 0 \quad\mbox{and}\quad\!  \rot \EE^+_{2}\times\nn = 0   
   \quad &\mbox{on}\quad \Gamma.
   \end{array}
    \right.
\end{equation}
Using \eqref{Efke0} and \eqref{cW1cd}, we obtain: 
\begin{equation}
\label{E2rotExnbis}
\left(\rot\EE^+_2\times\nn\right)_{\alpha}= -\lambda \sE_{1,\alpha}+ \cH\,\sE_{0,\alpha}   -b^{\sigma}_{\alpha} \sE_{0,\sigma}  \quad \mbox{on}\quad \Sigma \ ,
\end{equation}
We infer the boundary condition: 
\begin{equation}
\label{E2rotExn}
\rot\EE^+_2\times\nn= -\lambda (\nn \times\bsE_{1})\times\nn + \left(\cH-\cC \right) \bsE_{0}  \quad \mbox{on}\quad \Sigma \ .
\end{equation}
Here $(\cC  \bsE_{0})_{\alpha}=b^{\sigma}_{\alpha} \sE_{0,\sigma}$.

\newpage

\bibliographystyle{plain}
\bibliography{biblio}

\def\cprime{$'$}
\begin{thebibliography}{10}

\bibitem{AP23}
D.~Abou El Nasser El~Yafi and V.~P{\'e}ron.
\newblock Efficient asymptotic models for axisymmetric eddy current problems in
  linear ferromagnetic materials.
\newblock {\em Computers \& Mathematics with Applications}, 151:335--345, 2023.

\bibitem{APPK21}
D.~Abou El Nasser El~Yafi, V.~P{\'e}ron, R.~Perrussel, and
  L.~Kr{\"a}henb{\"u}hl.
\newblock Numerical study of the magnetic skin effect: Efficient
  parameterization of 2d surface-impedance solutions for linear ferromagnetic
  materials.
\newblock {\em International Journal of Numerical Modelling: Electronic
  Networks, Devices and Fields}, 36(3):e3051, 2023.

\bibitem{ammari1998effective}
H.~Ammari and S.~He.
\newblock Effective impedance boundary conditions for an inhomogeneous thin
  layer on a curved metallic surface.
\newblock {\em IEEE Transactions on Antennas and Propagation}, 46(5):710--715,
  1998.

\bibitem{artola1992diffraction}
M.~Artola and M.~Cessenat.
\newblock Diffraction d'une onde {\'e}lectromagn{\'e}tique par un obstacle
  born{\'e} {\`a} permittivit{\'e} et perm{\'e}abilit{\'e} {\'e}lev{\'e}es.
\newblock {\em Comptes rendus de l'Acad{\'e}mie des sciences. S{\'e}rie 1,
  Math{\'e}matique}, 314(5):349--354, 1992.

\bibitem{auvray2018asymptotic}
A.~Auvray and G.~Vial.
\newblock Asymptotic expansions and effective boundary conditions: a short
  review for smooth and nonsmooth geometries with thin layers.
\newblock {\em ESAIM: Proceedings and Surveys}, 61:38--54, 2018.

\bibitem{BIJ20}
G.~Beck, S.~Imperiale, and P.~Joly.
\newblock Asymptotic modelling of skin-effects in coaxial cables.
\newblock {\em SN Partial Differential Equations and Applications}, 1:1--34,
  2020.

\bibitem{bendali1996effect}
A.~Bendali and K.~Lemrabet.
\newblock The effect of a thin coating on the scattering of a time-harmonic
  wave for the helmholtz equation.
\newblock {\em SIAM Journal on Applied Mathematics}, 56(6):1664--1693, 1996.

\bibitem{CDFP11}
{G}. {C}aloz, {M}. {D}auge, {E}. {F}aou, and V.~P\'eron.
\newblock On the influence of the geometry on skin effect in electromagnetism.
\newblock {\em Computer Methods in Applied Mechanics and Engineering},
  200(9-12):1053--1068, 2011.

\bibitem{CDP09}
G.~Caloz, M.~Dauge, and V.~P\'eron.
\newblock Uniform estimates for transmission problems with high contrast in
  heat conduction and electromagnetism.
\newblock {\em Journal of Mathematical Analysis and Applications},
  370(2):555--572, 2010.

\bibitem{CDP19}
J.~Chabassier, M.~Durufl{\'e}, and V.~P\'eron.
\newblock {Equivalent boundary conditions for acoustic media with exponential
  densities. Application to the atmosphere in helioseismology}.
\newblock {\em {Applied Mathematics and Computation}}, 361:177--197, November
  2019.

\bibitem{CoDaNi10}
M.~Costabel, M.~Dauge, and S.~Nicaise.
\newblock Corner singularities and analytic regularity for linear elliptic
  systems. part i: Smooth domains.
\newblock 2010.

\bibitem{DFP10}
M.~Dauge, E.~Faou, and V.~P{\'e}ron.
\newblock Comportement asymptotique {\`a} haute conductivit{\'e} de
  l'{\'e}paisseur de peau en {\'e}lectromagn{\'e}tisme.
\newblock {\em Comptes Rendus. Math{\'e}matique}, 348(7-8):385--390, 2010.

\bibitem{Fa02}
E.~Faou.
\newblock Elasticity on a thin shell: formal series solution.
\newblock {\em Asymptot. Anal.}, 31(3-4):317--361, 2002.

\bibitem{HJN08}
H.~Haddar, P.~Joly, and H.-M. Nguyen.
\newblock Generalized impedance boundary conditions for scattering problems
  from strongly absorbing obstacles: the case of {M}axwell's equations.
\newblock {\em Math. Models Methods Appl. Sci.}, 18(10):1787--1827, 2008.

\bibitem{hoppe2018impedance}
D.~J Hoppe.
\newblock {\em Impedance boundary conditions in electromagnetics}.
\newblock CRC Press, 2018.

\bibitem{lafitte1993equations}
O.~Lafitte and G.~Lebeau.
\newblock {\'E}quations de maxwell et op{\'e}rateur d'imp{\'e}dance sur le bord
  d'un obstacle convexe absorbant.
\newblock {\em Comptes rendus de l'Acad{\'e}mie des sciences. S{\'e}rie 1,
  Math{\'e}matique}, 316(11):1177--1182, 1993.

\bibitem{MS84}
R.~C. MacCamy and E.~Stephan.
\newblock Solution procedures for three-dimensional eddy current problems.
\newblock {\em J. Math. Anal. Appl.}, 101(2):348--379, 1984.

\bibitem{MS85}
R.~C. MacCamy and E.~Stephan.
\newblock A skin effect approximation for eddy current problems.
\newblock {\em Arch. Rational Mech. Anal.}, 90(1):87--98, 1985.

\bibitem{Mo03}
P.~Monk.
\newblock {\em Finite element methods for {M}axwell's equations}.
\newblock Numerical Mathematics and Scientific Computation. Oxford University
  Press, New York, 2003.

\bibitem{nedelec2001}
J.-C. N{\'e}d{\'e}lec.
\newblock {\em Acoustic and electromagnetic equations: integral representations
  for harmonic problems}, volume 144.
\newblock Springer, 2001.

\bibitem{ECCOMAS24}
V.~P{\'e}ron.
\newblock {On a Skin Effect in Magnetic Conductors}.
\newblock In {\em {9th European Congress on Computational Methods in Applied
  Sciences and Engineering (ECCOMAS 2024)}}, Lisbonne, Portugal, June 2024.

\bibitem{peron2024uniform}
V.~P{\'e}ron.
\newblock Uniform estimates for transmission problems in electromagnetism with
  high contrast in magnetic permeabilities.
\newblock {\em arXiv preprint arXiv:2406.15276}, 2024.

\bibitem{PePo21}
V.~P{\'e}ron and C.~Poignard.
\newblock On a magnetic skin effect in eddy current problems: the magnetic
  potential in magnetically soft materials.
\newblock {\em Zeitschrift f{\"u}r Angewandte Mathematik und Physik}, 72(4),
  2021.

\bibitem{senior1995approximate}
T.~BA Senior and J.~L. Volakis.
\newblock {\em Approximate boundary conditions in electromagnetics}.
\newblock Number~41. Iet, 1995.

\bibitem{S83}
E.~Stephan.
\newblock Solution procedures for interface problems in acoustics and
  electromagnetics.
\newblock In {\em Theoretical acoustics and numerical techniques}, volume 277
  of {\em CISM Courses and Lectures}, pages 291--348. Springer, Vienna, 1983.

\bibitem{stupfel2021well}
B.~Stupfel, P.~Payen, and O.~Lafitte.
\newblock A well-posed and effective high-order impedance boundary condition
  for the time-harmonic scattering problem from a multilayer coated 3-d object.
\newblock {\em Progress In Electromagnetics Research B}, 94:127--144, 2021.

\end{thebibliography}

\end{document}